\documentclass{amsart}
\usepackage{amssymb}
\usepackage{amsmath}
\usepackage{latexsym}
\usepackage{graphicx}
\usepackage{amscd}
\usepackage{mathrsfs}
\usepackage[english]{babel}
\usepackage{hyperref}
\usepackage{color}
\usepackage{verbatim}
\usepackage{tikz}
\usetikzlibrary{math}

\input xy
\xyoption{all}

\newtheorem{theorem}[equation]{Theorem}

\newtheorem{prop}[equation]{Proposition}
\newtheorem{lemma}[equation]{Lemma}
\newtheorem{cor}[equation]{Corollary}

\theoremstyle{definition}
\newtheorem{definition}[equation]{Definition}

\newtheorem{example}[equation]{Example}

\newtheorem{remark}[equation]{Remark}

\newtheorem*{theorem*}{Main Theorem}

\newcommand{\N}{\ensuremath{\mathbb{N}}}
\newcommand{\Z}{\ensuremath{\mathbb{Z}}}
\newcommand{\Q}{\ensuremath{\mathbb{Q}}}

\newcommand{\R}{\ensuremath{\mathbb{R}}}

\newcommand{\Pro}{\ensuremath{\mathbb{P}}}
\newcommand{\A}{\ensuremath{\mathbb{A}}}

\newcommand{\cX}{\ensuremath{\mathscr{X}}}

\newcommand{\cC}{\ensuremath{\mathscr{C}}}

\newcommand{\cE}{\ensuremath{\mathscr{E}}}

\newcommand{\cU}{\ensuremath{\mathscr{U}}}
\newcommand{\cV}{\ensuremath{\mathscr{V}}}
\newcommand{\cW}{\ensuremath{\mathscr{W}}}
\newcommand{\cY}{\ensuremath{\mathscr{Y}}}

\newcommand{\Spec}{\ensuremath{\mathrm{Spec}\,}}

\newcommand{\ord}{\mathrm{ord}}
\newcommand{\red}{\mathrm{red}}

\newcommand{\lcm}{\mathrm{lcm}}

\newcommand{\Hom}{\mathrm{Hom}}
\newcommand{\coker}{\mathrm{coker}}

\newcommand{\an}{\mathrm{an}}
\newcommand{\spe}{\mathrm{sp}}

\newcommand{\ur}{\mathrm{ur}}
\newcommand{\Sh}{\mathrm{Sh}}
\newcommand{\Sk}{\mathrm{Sk}}
\newcommand{\wt}{\mathrm{wt}}
\newcommand{\Sm}{\mathrm{Sm}}
\newcommand{\Leb}{\lambda}

\newcommand{\tdeg}{\mathrm{tdeg}}
\newcommand{\st}{\mathrm{s}}

\newcommand{\ordmin}{\mathrm{ord}_{\min}}
\newcommand{\form}{\theta}
\newcommand{\pr}{\mathrm{pr}}
\newcommand{\unif}{\varpi}
\newcommand{\gp}{\mathrm{gp}}

\numberwithin{equation}{subsection}

\author{Mattias Jonsson}
\address{Dept of Mathematics, University of Michigan, Ann Arbor, MI 48109-1043, USA}
\email{mattiasj@umich.edu}

\author{Johannes Nicaise}
\address{Imperial College,
Department of Mathematics, South Kensington Campus,
London SW72AZ, UK, and KU Leuven, Department of Mathematics, Celestijnenlaan 200B, 3001 Heverlee, Belgium.} \email{j.nicaise@imperial.ac.uk}

\begin{document}
\title[Convergence of $P$-adic pluricanonical measures]{Convergence of $P$-adic pluricanonical measures to Lebesgue measures on skeleta in Berkovich spaces}

\begin{abstract}
Let $K$ be a non-archimedean local field, $X$ a smooth and proper $K$-scheme, and fix a pluricanonical form on $X$. 
For every finite extension $K'$ of $K$, the pluricanonical form induces a measure on the $K'$-analytic manifold $X(K')$.
We prove that, when $K'$ runs through all finite tame extensions of $K$, suitable normalizations of the pushforwards of these measures to the Berkovich analytification of $X$ converge to a Lebesgue-type measure on the temperate part of the Kontsevich--Soibelman skeleton, assuming the existence of a strict normal crossings model for $X$. We also prove a similar result for all finite extensions $K'$ under the assumption that $X$ has a log smooth model. This is a non-archimedean counterpart of analogous results for volume forms on degenerating complex Calabi--Yau manifolds by Boucksom and the first-named author. Along the way, we develop a general theory of Lebesgue measures on Berkovich skeleta over discretely valued fields.
\end{abstract}

\maketitle

\section{Introduction}\label{sec:intro}
Let $K$ be a local field, $X$ a smooth and proper $K$-scheme of pure dimension, and $\form$ a pluricanonical form on $X$. A case of particular interest is when $X$ has trivial canonical bundle and $\form$ is a nonzero global section of the latter. The set of $K$-points on $X$ has a natural structure of a $K$-analytic manifold, and the pluricanonical form $\form$ induces a 
{\em pluricanonical measure} on $X(K)$ that we denote by $|\form|$ (see Section~\ref{ss:pluri}).

Now assume that $K$ is non-archimedean, that is, a field of Laurent series over a finite field, or a finite extension of $\Q_p$ for some prime $p$. For every finite extension $K'$ of $K$, there is a natural continuous map
$\pi_{K'}\colon X(K')\to X^{\an}$
from the $K'$-analytic manifold $X(K')$ to the Berkovich analytification $X^{\an}$ of $X$. The aim of this paper is to study the asymptotic behavior of the pushforward
\begin{equation*}
  m_{K'}:=(\pi_{K'})_*|\form\otimes_K K'|
\end{equation*}
of the measure $|\form\otimes_K K'|$ to $X^{\an}$ as $K'$ runs through
finite extensions of $K$.

Fix an algebraic closure $K^a$ of $K$, and let $\mathcal{E}^a_K$ be the set of finite extensions $K'$ of $K$ in $K^a$, ordered by inclusion.
Let $\mathcal{E}^{\ur}_K$ be the subset of $\mathcal{E}^a_K$
consisting of unramified extensions,
and let $K^{\ur}$ be the union of all extensions in $\mathcal{E}^{\ur}_K$ (that is, the maximal unramified extension of $K$ in $K^a$).
\begin{theorem*}
  Assume $X$ admits a log smooth model over the valuation ring $R$ in $K$. Then there exist Lebesgue-type measures $\Leb^a$ and $\Leb^{\ur}$ on $X^{\an}$,  and positive constants $c_{K'}^a$ and $c_{K'}^{\ur}$ for $K'$ in $\mathcal{E}^a_K$ and $\mathcal{E}^{\ur}_K$, respectively, such that
  \begin{equation*}
    \lim_{K'\in\mathcal{E}^a_K}c_{K'}^a\,m_{K'}=\Leb^a,
    \quad\text{and}\quad
    \lim_{K'\in\mathcal{E}^{\ur}_K}c^{\ur}_{K'}\,m_{K'}=\Leb^{\ur}
  \end{equation*}
  in the weak sense of positive Radon measures on $X^{\an}$.
\end{theorem*}
This result is an amalgam of Theorems~\ref{thm:shi} and~\ref{thm:padic}, where the reader can find more precise statements. For the sake of this introduction, we only make the following remarks.

First, the measure $\Leb^a$ is supported on the top-dimensional part of the {\em Kontsevich--Soibelman skeleton} $\Sk(X,\form)\subset X^{\an}$ of $(X,\form)$~\cite{KS,MuNi}, defined as the closure in $X^{\an}$ of the locus where a certain function, the \emph{weight function} $\wt_\theta$ attains its infimum, see Section~\ref{ss:KS}.
If $X$ has a log smooth (or merely log regular) model, then $\Sk(X,\theta)$ is nonempty and carries a natural piecewise integral affine structure, and thus a natural Lebesgue measure (see Section~\ref{ss:lebesgue}); the measure $\Leb^a$ equals the Lebesgue measure on the top dimensional part of $\Sk(X,\theta)$, suitably
weighted to make it well-behaved under tame finite extensions of the base field $K$; we call it the \emph{stable} Lebesgue measure. The constants $c_{K'}^a$ in the Main Theorem are given by
\begin{equation*}
  c_{K'}^a=\frac{q^{w[K':K]}}{e^d},
\end{equation*}
where $q$ is the cardinality of the residue field on $K$, $e=e(K'/K)$ is the ramification index, $w=\inf\wt_\theta$, and $d$ is the dimension of $\Sk(X,\theta)$.

Second, the unramified version of the convergence result in the Main Theorem is of course only interesting when $X(K^{\ur})\ne\emptyset$, but it holds without the assumption of the existence of a log smooth model of $X$. The measure $\Leb^{\ur}$ is a finite atomic measure supported on the \emph{Shilov boundary} $\Sh(X,\theta)$ of $(X,\theta)$, see Section~\ref{sec:shilov}.
We have
\begin{equation*}
  c_{K'}^{\ur}=q^{w^{\ur}[K':K]},
\end{equation*}
where $w^{\ur}$ is the minimum of $\wt_\theta$ over the points in $X^{\an}$ whose residue field is unramified over $K$.
If $X$ admits a log regular model, then $\Sh(X,\theta)$ is equal to
the set of integral points of $\Sk(X,\theta)$.

 It is generally believed that $X$ admits a log smooth model after base change to a finite extension of $K$. Without such a base change, one can only hope for the existence of a log regular model (this is equivalent to the existence of a strict normal crossings model, which would follow from embedded resolution of singularities over the excellent ring $R$). We establish a variant of our main theorem under this weaker assumption by considering tame extensions of $K$.
  Let $\mathcal{E}^t_K\subset\mathcal{E}^a_K$ be the set of tame finite extensions of $K$ in $K^a$.
Assuming that $X$ admits a log regular model over $R$, we show that
\begin{equation*}
  \lim_{K'\in\mathcal{E}^t_K}c_{K'}^t m_{K'}=\Leb^t,
\end{equation*}
where $\Leb^t$ is the restriction of $\Leb^a$ to the  closure $\Sk^t(X,\form)$ of the set of $\Z_{(p)}$-integral points of
$\Sk(X,\theta)$, and $c_{K'}^t$ is defined in the same way as $c^a_{K'}$, replacing $d$ by the dimension of $\Sk^t(X,\theta)$.
 When $X$ has semistable reduction, the measures $\Leb^a$ and $\Leb^t$ are both equal to (unweighted) Lebesgue measure on the top-dimensional part of $\Sk(X,\theta)$, whereas $\Leb^{\ur}$ is a sum of Dirac masses on the
integral points of $\Sk(X,\theta)$.
 A variant of our theorem for non-local fields is given in Theorem~\ref{thm:main}.

We deduce these results from a general study of Lebesgue measures on skeleta over arbitrary discretely valued fields, combined with the theory of weak N\'eron models and the Lang--Weil estimates. Our theory of Lebesgue measures on skeleta is of independent interest, and generalizes the natural metric on the dual reduction graph of a smooth and proper curve over $K$ (see Section~\ref{sec:curve}).

\subsection*{Example: the Tate elliptic curve}
Let us illustrate the Main Theorem for the example of the Tate elliptic curve given by
\begin{equation}\label{eq:Tate}
  X=\{xyz+2(x^3+y^3+z^3)=0\}\subset\Pro^2_{\Q_2}\tag{*}
\end{equation}
over $K=\Q_2$. This simple example illustrates some of the phenomena in the general case.

Equation~(*) defines a semistable model $\cX$ over $R=\Z_2$ with trivial relative canonical bundle $\omega_{\cX/R}$. Let $\form$ be a generator of
$\omega_{\cX/R}$.
 The Kontsevich-Soibelman skeleton $\Sk(X,\form)$ coincides with the skeleton $\Sk(\cX)$  of $\cX$ and is homeomorphic to a circle. The Shilov boundary $\Sh(X,\theta)$ consists of the three vertices of $\Sk(\cX)$,
 corresponding to the irreducible components of the special fiber of $\cX$. See Figure~\ref{F101}. We have a canonical continuous retraction
 $$\rho_{\cX}\colon X^{\an}\to \Sk(\cX)=\Sk(X,\theta).$$
The measure $m_K$ is supported on the Cantor set $X(K)\subset X^{\an}$, which is isomorphic as a $K$-analytic manifold to a union of three open unit balls in $K$ (the fibers of the three smooth $k$-rational points on $\cX_k$ under the reduction map $X(K)\to \cX_k(k)$). This isomorphism identifies  $|\theta|$ with the Haar measure with mass $1/2$ on each of these open balls.
 The pushforward $(\rho_{\cX})_*m_K$ is an atomic measure on $\Sh(X,\theta)$, giving mass $1/2$ to each point.
\begin{figure}[ht]
  \centering
  \begin{tikzpicture}
    \draw[thick,red] (0,0) ellipse (1.0cm and 1.0cm);

    \tikzmath{\scpar = 0.5;}
    \tikzmath{\dang = 60;}
    \tikzmath{\x1 = 0; \y1 =0;}
    \tikzmath{\scale=1; \ang=0; }
    \tikzmath{\dotsize=0.8; }

    \foreach \i in {0,1,2} {%
      \tikzmath{\ang=90+ \i*120;}
      \tikzmath{\x1 = \x1+\scale*cos(\ang); \y1 = \y1+\scale*sin(\ang);}
      \draw [shift={(\x1,\y1)}, scale=\scale, rotate=\ang](0,0) .. controls (1/3,-0.1) and (2/3,0.1) .. (1,0);
      \node at (\x1,\y1) [circle,fill,inner sep=1.0pt]{};
      \foreach \j in {-1,1} {%
        \tikzmath{\x1 = \x1+\scale*cos(\ang); \y1 = \y1+\scale*sin(\ang);}
        \tikzmath{\scale=\scale*\scpar;}
        \tikzmath{\ang=\ang+ \j*\dang;}
        \draw [shift={(\x1,\y1)}, scale=\scale, rotate=\ang](0,0) .. controls (1/3,-0.1) and (2/3,0.1) .. (1,0);
        \foreach \k in {-1,1} {%
          \tikzmath{\x1 = \x1+\scale*cos(\ang); \y1 = \y1+\scale*sin(\ang);}
          \tikzmath{\scale=\scale*\scpar;}
          \tikzmath{\ang=\ang+ \k*\dang;}
          \draw [shift={(\x1,\y1)}, scale=\scale, rotate=\ang](0,0) .. controls (1/3,-0.1) and (2/3,0.1) .. (1,0);
          \foreach \l in {-1,1} {%
            \tikzmath{\x1 = \x1+\scale*cos(\ang); \y1 = \y1+\scale*sin(\ang);}
            \tikzmath{\scale=\scale*\scpar;}
            \tikzmath{\ang=\ang+ \l*\dang;}
            \draw [shift={(\x1,\y1)}, scale=\scale, rotate=\ang](0,0) .. controls (1/3,-0.1) and (2/3,0.1) .. (1,0);
            \tikzmath{\x2=\x1+\scale*cos(\ang); \y2=\y1+\scale*sin(\ang);}
            \node at (\x2,\y2) [circle,fill,inner sep=\dotsize]{};
          }
        }
      }
    }
  \end{tikzpicture}
  \caption{The picture shows the Berkovich analytification $X^{\an}$ of the Tate elliptic curve $X$ over $\Q_2$ defined by~\eqref{eq:Tate}. The circle is the Kontsevich-Soibelman skeleton $\Sk(X,\theta)$, and the three vertices on the circle form the Shilov boundary $\Sh(X,\theta)$. The measure $m_K$ induced by $|\form|$ is supported on the Cantor set formed by the endpoints.}\label{F101}
\end{figure}
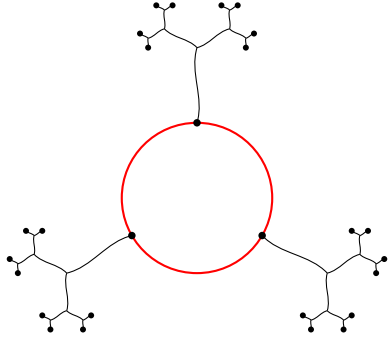

\medskip
Now consider a finite unramified extension $K'/K$. The image of $X(K')$ in $X^{\an}$ is a Cantor set that still retracts to $\Sh(X,\theta)$, so $r_*m_{K'}$ is still a finite atomic measure supported on $\Sh(X,\theta)$, see Figure~\ref{F102}.
However, any connected component of $X^{\an}\setminus\Sh(X,\theta)$ only carries a small fraction of the total mass of $m_{K'}$.
This explains why the limit measure $\Leb^{\ur}$ in the Main Theorem is a sum of Dirac masses at the three points in the Shilov boundary of $(X,\theta)$.
\begin{figure}[ht]
  \centering
  \begin{tikzpicture}
    \draw[thick,red] (0,0) ellipse (1.0cm and 1.0cm);

    \tikzmath{\scpar = 0.4;}
    \tikzmath{\dang = 100;}
    \tikzmath{\x1 = 0; \y1 =0;}
    \tikzmath{\scale=1; \ang=0; }
    \tikzmath{\dotsize=0.7; }

    \foreach \i in {0,1,2} {%
      \tikzmath{\ang=90+ \i*120;}
      \tikzmath{\x1 = \x1+\scale*cos(\ang); \y1 = \y1+\scale*sin(\ang);}
      \node at (\x1,\y1) [circle,fill,inner sep=1.0pt]{};
      \foreach \j in {-1,1} {%
        \tikzmath{\ang=\ang+ \j*50;}
        \draw [shift={(\x1,\y1)}, scale=\scale, rotate=\ang](0,0) .. controls (1/3,-0.1) and (2/3,0.1) .. (1,0);
        \foreach \k in {-1,0,1} {%
          \tikzmath{\x1 = \x1+\scale*cos(\ang); \y1 = \y1+\scale*sin(\ang);}
          \tikzmath{\scale=\scale*\scpar;}
          \tikzmath{\ang=\ang+ \k*\dang;}
          \draw [shift={(\x1,\y1)}, scale=\scale, rotate=\ang](0,0) .. controls (1/3,-0.1) and (2/3,0.1) .. (1,0);
          \foreach \l in {-1,0,1} {%
            \tikzmath{\x1 = \x1+\scale*cos(\ang); \y1 = \y1+\scale*sin(\ang);}
            \tikzmath{\scale=\scale*\scpar;}
            \tikzmath{\ang=\ang+ \l*\dang;}
            \draw [shift={(\x1,\y1)}, scale=\scale, rotate=\ang](0,0) .. controls (1/3,-0.1) and (2/3,0.1) .. (1,0);
            \foreach \m in {-1,0,1} {%
              \tikzmath{\x1 = \x1+\scale*cos(\ang); \y1 = \y1+\scale*sin(\ang);}
              \tikzmath{\scale=\scale*\scpar;}
              \tikzmath{\ang=\ang+ \m*\dang;}
              \draw [shift={(\x1,\y1)}, scale=\scale, rotate=\ang](0,0) .. controls (1/3,-0.1) and (2/3,0.1) .. (1,0);
              \tikzmath{\x2=\x1+\scale*cos(\ang); \y2=\y1+\scale*sin(\ang);}
              \node at (\x2,\y2) [circle,fill,inner sep=\dotsize]{};
            }
          }
        }
      }
    }
  \end{tikzpicture}
  \caption{The endpoints in the picture is a Cantor set that equals the support of the measure $m_{K'}$ on $X^{\an}$, where $K'$ is an unramified extension of $\Q_2$ of degree two.}\label{F102}
\end{figure}
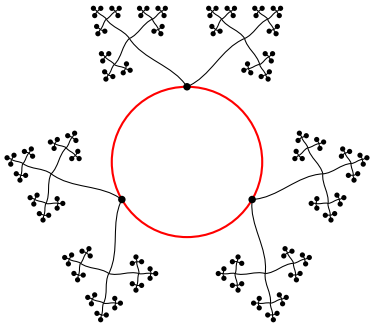

\medskip
Finally consider a ramified extension $K'/K$.
The image of $X(K')$ in $X^{\an}$ is then a Cantor set that retracts to a finite  but possibly large subset of $\Sk(X,\theta)$, see Figure~\ref{F103}. Any connected component of $X^{\an}\setminus\Sk(X,\theta)$ carries only a small fraction of the mass of $m_{K'}$ if $[K':K]$ is large.
This explains why the limit measure $\Leb^a$ in the Main Theorem is equal to Lebesgue measure on the circle $\Sk(X,\theta)$.

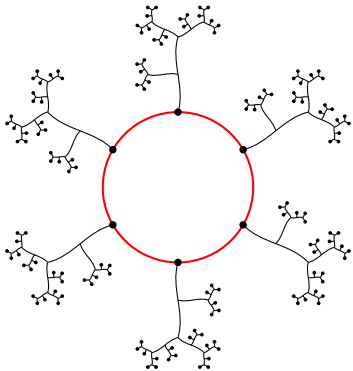
\begin{figure}[ht]
  \centering
  \begin{tikzpicture}
    \draw[thick, red] (0,0) ellipse (1.0cm and 1.0cm);

    \tikzmath{\scpar = 0.4;}
    \tikzmath{\dang = 60;}
    \tikzmath{\x1 = 0; \y1 =0;}
    \tikzmath{\scale=1; \ang=0; }
    \tikzmath{\newbrpar=0.5; }
    \tikzmath{\dotsize=0.5; }

    \foreach \i in {0,1,2,3,4,5} {
      \tikzmath{\ang=90+ \i*60;}
      \tikzmath{\x1 = \x1+\scale*cos(\ang); \y1 = \y1+\scale*sin(\ang);}
      \node at (\x1,\y1) [circle,fill,inner sep=1.0pt]{};
      \draw [shift={(\x1,\y1)}, scale=\scale, rotate=\ang](0,0) .. controls (1/3,-0.1) and (2/3,0.1) .. (1,0);
      \foreach \j in {-1,1} {
        \tikzmath{\x1 = \x1+\scale*cos(\ang); \y1 = \y1+\scale*sin(\ang);}
        \tikzmath{\scale=\scale*\scpar;}
        \tikzmath{\ang=\ang+ \j*\dang;}
        \draw [shift={(\x1,\y1)}, scale=\scale, rotate=\ang](0,0) .. controls (1/3,-0.1) and (2/3,0.1) .. (1,0);
        \foreach \k in {-1,1} {
          \tikzmath{\x1 = \x1+\scale*cos(\ang); \y1 = \y1+\scale*sin(\ang);}
          \tikzmath{\scale=\scale*\scpar;}
          \tikzmath{\ang=\ang+ \k*\dang;}
          \draw [shift={(\x1,\y1)}, scale=\scale, rotate=\ang](0,0) .. controls (1/3,-0.1) and (2/3,0.1) .. (1,0);
          \foreach \l in {-1,1} {
            \tikzmath{\x1 = \x1+\scale*cos(\ang); \y1 = \y1+\scale*sin(\ang);}
            \tikzmath{\scale=\scale*\scpar;}
            \tikzmath{\ang=\ang+ \l*\dang;}
            \draw [shift={(\x1,\y1)}, scale=\scale, rotate=\ang](0,0) .. controls (1/3,-0.1) and (2/3,0.1) .. (1,0);
            \tikzmath{\x2=\x1+\scale*cos(\ang); \y2=\y1+\scale*sin(\ang);}
            \node at (\x2,\y2) [circle,fill,inner sep=\dotsize]{};
          } 
          \tikzmath{\x1 = \x1+\newbrpar*\scale*cos(\ang); \y1 = \y1+\newbrpar*\scale*sin(\ang);}
          \tikzmath{\scale=\scale*\scpar;}
          \tikzmath{\ang=\ang+ 90;}
          \draw [shift={(\x1,\y1)}, scale=\scale, rotate=\ang](0,0) .. controls (1/3,-0.1) and (2/3,0.1) .. (1,0);
          \tikzmath{\x2=\x1+\scale*cos(\ang); \y2=\y1+\scale*sin(\ang);}
          \node at (\x2,\y2) [circle,fill,inner sep=\dotsize]{};
        } 
        \tikzmath{\x1 = \x1+0.5*\scale*cos(\ang); \y1 = \y1+0.5*\scale*sin(\ang);}
        \tikzmath{\scale=\scale*\scpar;}
        \tikzmath{\ang=\ang+ 90;}
        \draw [shift={(\x1,\y1)}, scale=\scale, rotate=\ang](0,0) .. controls (1/3,-0.1) and (2/3,0.1) .. (1,0);
        \foreach \l in {-1,1} {
          \tikzmath{\x1 = \x1+\scale*cos(\ang); \y1 = \y1+\scale*sin(\ang);}
          \tikzmath{\scale=\scale*\scpar;}
          \tikzmath{\ang=\ang+ \l*\dang;}
          \draw [shift={(\x1,\y1)}, scale=\scale, rotate=\ang](0,0) .. controls (1/3,-0.1) and (2/3,0.1) .. (1,0);
          \tikzmath{\x2=\x1+\scale*cos(\ang); \y2=\y1+\scale*sin(\ang);}
          \node at (\x2,\y2) [circle,fill,inner sep=\dotsize]{};
        } 
        \tikzmath{\x1 = \x1+\newbrpar*\scale*cos(\ang); \y1 = \y1+\newbrpar*\scale*sin(\ang);}
        \tikzmath{\scale=\scale*\scpar;}
        \tikzmath{\ang=\ang+ 90;}
        \draw [shift={(\x1,\y1)}, scale=\scale, rotate=\ang](0,0) .. controls (1/3,-0.1) and (2/3,0.1) .. (1,0);
        \tikzmath{\x2=\x1+\scale*cos(\ang); \y2=\y1+\scale*sin(\ang);}
        \node at (\x2,\y2) [circle,fill,inner sep=\dotsize]{};
      } 
      \tikzmath{\x1 = \x1+\newbrpar*\scale*cos(\ang); \y1 = \y1+\newbrpar*\scale*sin(\ang);}
      \tikzmath{\scale=\scale*\scpar;}
      \tikzmath{\ang=\ang+ 90;}
      \draw [shift={(\x1,\y1)}, scale=\scale, rotate=\ang](0,0) .. controls (1/3,-0.1) and (2/3,0.1) .. (1,0);

      \foreach \k in {-1,1} {
        \tikzmath{\x1 = \x1+\scale*cos(\ang); \y1 = \y1+\scale*sin(\ang);}
        \tikzmath{\scale=\scale*\scpar;}
        \tikzmath{\ang=\ang+ \k*\dang;}
        \draw [shift={(\x1,\y1)}, scale=\scale, rotate=\ang](0,0) .. controls (1/3,-0.1) and (2/3,0.1) .. (1,0);
        \foreach \l in {-1,1} {
          \tikzmath{\x1 = \x1+\scale*cos(\ang); \y1 = \y1+\scale*sin(\ang);}
          \tikzmath{\scale=\scale*\scpar;}
          \tikzmath{\ang=\ang+ \l*\dang;}
          \draw [shift={(\x1,\y1)}, scale=\scale, rotate=\ang](0,0) .. controls (1/3,-0.1) and (2/3,0.1) .. (1,0);
          \tikzmath{\x2=\x1+\scale*cos(\ang); \y2=\y1+\scale*sin(\ang);}
          \node at (\x2,\y2) [circle,fill,inner sep=\dotsize]{};
        } 
        \tikzmath{\x1 = \x1+\newbrpar*\scale*cos(\ang); \y1 = \y1+\newbrpar*\scale*sin(\ang);}
        \tikzmath{\scale=\scale*\scpar;}
        \tikzmath{\ang=\ang+ 90;}
        \draw [shift={(\x1,\y1)}, scale=\scale, rotate=\ang](0,0) .. controls (1/3,-0.1) and (2/3,0.1) .. (1,0);
        \tikzmath{\x2=\x1+\scale*cos(\ang); \y2=\y1+\scale*sin(\ang);}
        \node at (\x2,\y2) [circle,fill,inner sep=\dotsize]{};
      }
    }
  \end{tikzpicture}
    \caption{The endpoints in the picture is a Cantor set that equals the support of the measure $m_{K'}$ on $X^{\an}$, where $K'$ is a ramified extension of $\Q_2$ of degree two.}\label{F103}
\end{figure}

\subsection*{Related results}

\subsubsection*{The Lang--Weil estimates}
Our convergence results can be viewed as a non-archimedean analog of the Lang--Weil estimates
for varieties over finite fields. Let $F$ be a finite field, and let $X$ be an integral proper $F$-scheme, of dimension $n$. We endow $F$ with its trivial absolute value and denote by $X^{\an}$ the analytification of $X$. Let $\eta$ be the point of $X^{\an}$ corresponding to the trivial absolute value on the function field of $X$; this point plays the role of the skeleton of $X$.
 We denote by $q$ the cardinality of $F$, and we fix an algebraic closure $F^a$ of $F$. It follows easily from the Lang--Weil estimates that, as $F'$ ranges through the finite extensions of $F$ in $F^a$, ordered by inclusion, the pushforward to $X^{\an}$ of the normalized counting measure
$$q^{-n[F'\colon F]}\sum_{x\in X(F')}\delta_x$$
 on $X(F')$ converges to a weighted Dirac measure supported at $\eta$. The mass of this measure
 is equal to the number of geometric irreducible components of $X$. We will prove this result as a toy example in Section~\ref{sec:langweil}. In fact, the Lang--Weil estimates are a crucial ingredient for our convergence results over non-archimedean local fields; the passage from non-archimedean local measures to point counting over finite fields occurs through the theory of weak N\'eron models (see Section~\ref{sec:shilov}).

\subsubsection*{Convergence of complex volume forms}
An important motivation for this paper are analogous results on the convergence of volume forms on degenerating families $X\to\mathbb{D}^*$ of complex Calabi--Yau manifolds, obtained by S\'ebastien Boucksom and the first-named author in~\cite{BJ}. They can be interpreted as a  measure-theoretic version of a fundamental conjecture of Kontsevich and Soibelman~\cite{KS} on the collapse of Ricci-flat metrics in projective maximally degenerating families, motivated by mirror symmetry. In~\cite{BJ}, the limit measure lives on the Kontsevich-Soibelman skeleton, and the convergence takes place in a suitable hybrid space obtained from $X$ by inserting a Berkovich space as central fiber.
In the present work, the convergence takes place directly in a Berkovich space.

\subsubsection*{Igusa zeta functions and motivic measures}
If $K$ is a complete discretely valued field with infinite residue field, then $K$ is not locally compact and no longer carries a Haar measure. The natural generalization of Haar measures in this context is the theory of motivic integration as developed in~\cite{loeser-sebag}. The behavior of the motivic volume under finite extensions of $K$ is encoded in Denef and Loeser's motivic Igusa zeta function {\em via} its interpretation in~\cite{nicaise-sebag}. If $K$ has equal characteristic zero, the leading asymptotic term of the motivic zeta functions of Calabi-Yau varieties was studied in~\cite{halle-nicaise}, and the relations with the Kontsevich--Soibelman skeleton are highlighted in~\cite[3.2.3]{halle-nicaise}. These results should be viewed as a geometric version  of our convergence statements for non-archimedean local fields.
 Closely related results over $p$-adic fields had already appeared in the literature in a somewhat different setting, especially in the work of Chambert-Loir and Tschinkel~\cite{CLT}.

 \subsubsection*{Philippon's work}
 It was pointed out to us by Antoine Chambert-Loir that an early trace of our main results can already been found in Philippon's paper~\cite{philippon}. Let $K$ be a $p$-adic field, and denote by $q$ the cardinality of its residue field.
  For every finite extension $K'$ of $K$, denote by $\mu_{K'}$ the Haar measure on the closed unit polydisc in $K^n$.
  Then the formula at the bottom of page 1049 in~\cite{philippon} implies that, as $K'$ ranges through the finite extensions of $K$ in a fixed algebraic closure of $K$, the pushforward of the measure $\mu_K'$ to $\A^{n,\an}_K$ converges to the Dirac measure
  supported at the Gauss point of the closed unit polydisc in $\A^{n,\an}_K$.

\subsection*{Structure of the paper}
Section~\ref{sec:prelim} contains some preliminary results and facts about  Kontsevich--Soibelman skeleta and weak N\'eron models. In Section~\ref{sec:log}, we collect auxiliary results on logarithmic models to be used in the remainder of the paper.
The technical heart of the paper consists of Sections~\ref{sec:skeleta} and~\ref{sec:lebesgue}, where we define Lebesgue measures on Berkovich skeleta and study their properties under finite extensions of the base field $K$.
For technical reasons, it is convenient to work not only with strict normal crossings models, but, more generally, with log regular models, building upon the theory developed in~\cite{MuNi} and~\cite{BrMa}.
In Section~\ref{sec:padic}, we prove our convergence theorem for measures associated with pluricanonical forms over non-archimedean local fields (Theorems~\ref{thm:shi} and~\ref{thm:padic}). These provide a precise version of the Main Theorem above.
Finally, in Section~\ref{sec:convshilov} we prove an equidistribution result for Shilov boundary points associated with weak N\'eron models (Theorem~\ref{thm:main}), which can be viewed as a version of our main theorem over discretely valued fields that are not necessarily locally compact.

\subsection*{Acknowledgements} The origin of this paper is a question by Antoine Chambert-Loir, raised in a discussion with S\'ebastien Boucksom and the first-named author. We are grateful to him for sharing this question, and to S\'ebastien Boucksom and Joe Rabinoff for fruitful discussions.
Johannes Nicaise is supported by long term structural funding (Methusalem
grant) of the Flemish Government. Mattias Jonsson is supported by
The second author was partially supported by NSF grant DMS-1600011
and the United States---Israel Binational Science Foundation.
We thank these institutions as well as the Simons Foundation for organizing the symposium where this work first originated.

\section{Preliminaries}\label{sec:prelim}
\subsection{Notation and conventions}
We denote by $R$ a complete discrete valuation ring, with quotient field $K$ and residue field $k$. We denote by $p$ the characteristic exponent of $k$ and we choose a uniformizer $\unif$ in $R$. We also fix an algebraic closure $K^a$ of $K$, with valuation ring $R^a$, and we denote by $R^{\ur}$ the maximal unramified extension of $R$ in $R^{a}$.
 The residue field of $R^{\ur}$ is a separable closure of $k$, which we denote by $k^s$. We write $K^{\ur}$ for the quotient field of $R^{\ur}$.
  If $k$ is finite, then we will denote its cardinality
 by $q$.

  A finite extension $K'$ of $K$ is called {\em tame} (or {\em tamely ramified}) if the ramification degree $e(K'/K)$ of $K'$ over $K$ is prime to $p$, and the residue field of
  $K'$ is a separable extension of $k$. The union of all tame finite extensions of $K$ in $K^a$ is called the tame closure of $K$ in $K^a$, and denoted by $K^t$.

 We denote by $v_K$ the discrete valuation on $K$, normalized such that $v_K(\unif)=1$. We define a non-archimedean absolute value $|\cdot|_K$ on $K$ by setting
 $|a|_K=\varepsilon^{v_K(a)}$ for every $a\in K^{\times}$, where $\varepsilon$ is a fixed element in $(0,1)$. If $k$ is finite, we will take $\varepsilon=1/q$. Whenever $K'$ is a finite extension of $K$, we denote by $v_{K'}$ the unique valuation on $K'$ that extends $v_K$. The value group of $v_{K'}$ is equal to $(1/e)\Z$, where $e=e(K'/K)$ is the ramification index of $K'$ over $K$.
 We also endow $K'$ with the absolute value $|\cdot|_{K'}=\varepsilon^{v_{K'}(\cdot)}$; this is the unique absolute value that extends $|\cdot|_K$.

 For every $K$-scheme $X$ of finite type, we denote by $X^{\an}$ the Berkovich analytification of $X$~\cite{berk}. For every point $x$ on $X^{\an}$, we denote by $\mathscr{H}(x)$ the residue field of $X^{\an}$ at $x$. This is a complete valued field extension of $K$, whose residue field is denoted by  $\widetilde{\mathscr{H}}(x)$.

 If $X$ is a separated $K$-scheme of finite type, then an $R$-model for $X$ is a flat separated $R$-scheme of finite type $\cX$, endowed with an isomorphism of $K$-schemes $\cX\otimes_R K\to X$. Note that we do not impose any properness condition on $X$ or $\cX$. We denote by $\widehat{\cX}$ the formal $\unif$-adic completion of $\cX$ and by
 $\widehat{\cX}_\eta$ its generic fiber; this is a compact analytic domain in $X^{\an}$, and it is equal to $X^{\an}$ if and only if $\cX$ is proper over $R$. We denote by $$\spe_{\cX}:\widehat{\cX}_\eta\to \cX_k$$ the specialization morphism (also called reduction map) associated with $\cX$.

 If $X$ is regular and  proper over $K$, then a proper $R$-model $\cX$ of $X$ is called an $snc$-model if $\cX$ is regular and $\cX_k$ is a divisor with strict normal crossings (not necessarily reduced). The existence of $snc$-models is known when $k$ has characteristic zero (by Hironaka's resolution of singularities), and in arbitrary characteristic when the dimension of $X$ is at most $2$~\cite{CJS, cossart-piltant}.

 For the basic theory of piecewise integral affine structures, we refer to
 \cite[\S1]{BerkContractII}. We will use additive, rather than multiplicative, notation;
 for every integer $e>0$,
 what we call piecewise integral $(1/e)\Z$-affine spaces are called $R_S$-piecewise linear spaces in \cite[\S1]{BerkContractII}, with $R=\sqrt[e]{|K^{\times}|}$ and
 $S=\Z$.  Thus the piecewise models of
 piecewise integral $(1/e)\Z$-affine spaces are rational polytopes $P$ in $\R^n$ endowed with the group of functions of the form
 $$P\to \R\colon (x_1,\ldots,x_n)\mapsto a_1x_1+\ldots+a_nx_n+b$$ where the coefficients $a_i$ are integers and the constant term $b$ lies in $(1/e)\Z$.
   On any piecewise integral $(1/e)\Z$-affine space, it makes sense to speak of $(1/e')\Z$-integral points whenever $e'$ is an integer multiple of $e$: these are the points where every piecewise integral $(1/e)\Z$-affine function takes a value in $(1/e')\Z$.

 Whenever $K'$ is a finite extension of $K$ of ramification index $e$, we will also use the terminology {\em integral $K'$-affine}
 instead of integral $(1/e)\Z$-affine (in Berkovich's notation, this would be $|(K')^{\times}|_{\Z}$-affine).
 This will make it easier to keep track of the behavior of these structures under extensions of $K$.       Our primary examples of
  piecewise integral $K'$-affine spaces will be skeleta of log regular models of regular proper $K'$-schemes (see Section \ref{sec:skeleta}).

 \subsection{The Kontsevich--Soibelman skeleton}\label{ss:KS}
 Let $X$ be a connected smooth and proper $K$-scheme, and let $\form$ be a non-zero $m$-canonical form on $X$, for some positive integer $m$.
 It is explained in~\cite{MuNi} how one can attach to $\form$ a canonical subspace of $X^{\an}$, called the {\em Kontsevich--Soibelman skeleton} of the pair $(X,\form)$.
  Such an object first appeared in the  work of Kontsevich and Soibelman on the non-archimedean SYZ fibration~\cite{KS}. Let us briefly recall the construction. For every normal $R$-model $\cX$ of $X$ and every irreducible component $E$ of $\cX_k$, the fiber of the specialization map $\spe_{\cX}$ over the generic point $\xi$ of $E$ contains a unique point $x$. This point  is called the {\em divisorial point} associated with $\cX$ and $E$. If we denote by $N$ the multiplicity of $E$ in $\cX_k$, then  $x$ corresponds to the divisorial valuation $(1/N)\mathrm{ord}_E$ on the function field of $X$.

  The {\em weight} of $\form$ at $x$ is defined in the following way. Locally around $\xi$, the model $\cX$ is regular, so that we can consider the relative canonical line bundle $\omega_{\cX/R}$. We view $\form$ as a rational section of the logarithmic relative $m$-canonical bundle
  $\omega_{\cX/R}(\cX_{k,\red}-\cX_k)^{\otimes m}$ and we denote the associated Cartier divisor by $\mathrm{div}_{\cX}(\form)$. If we denote by $w$ the multiplicity of $\mathrm{div}_{\cX}(\form)$ at $\xi$, then the weight  of $\form$ at $x$ is defined as
  $$\wt_{\form}(x)=\frac{w}{mN}.$$
 The minimal weight $\wt_{\min}(X,\form)$ of $\form$ on $X$ is defined as
the infimum of the weights $\wt_{\form}(x)$ where $x$ runs through the set of divisorial points on $X^{\an}$, and the Kontsevich--Soibelman skeleton $\Sk(X,\form)$
 is the closure in $X^{\an}$ of the set of divisorial points $x$ such that $\wt_{\form}(x)=\wt_{\min}(X,\form)$. The {\em essential skeleton} $\Sk(X)$ is the union of the skeleta $\Sk(X,\form)$ over all non-zero
 pluricanonical forms $\form$ on $X$.

\begin{remark}\label{rema:weight}
In the definition of the weight function in~\cite{MuNi}, the reference line bundle was taken to be $\omega_{\cX/R}(\cX_{k,\red})^{\otimes m}$, rather than $\omega_{\cX/R}(\cX_{k,\red}-\cX_k)^{\otimes m}$, and the factor $m$ was not included in the denominator. Thus the weight function in~\cite{MuNi} is given  by the formula $\mathrm{wt}^{\mathrm{MN}}_{\form}=m(\wt_{\form}+1)$, so that the definition of the Kontsevich--Soibelman skeleton is not affected. The reason for our different choice of reference bundle is that it agrees with the relative canonical bundle for smooth models and behaves better under ramified extensions of $R$. The extra factor $m$ in the denominator appears naturally when we consider $p$-adic measures associated with $m$-canonical forms (see Section~\ref{ss:pluri}). If $k$ has characteristic zero, the weight function $\wt_{\form}$ is related to Temkin's metric $\|\form\|$ on $X^{\an}$ by the formula
$$m\cdot \wt_{\form}=-\log_{\varepsilon}\|\form\|,$$ see Theorem 8.3.3 in~\cite{temkin}. If $k$ has positive characteristic, an additional term appears in the comparison statement; this is caused by issues of wild ramification.
\end{remark}

 Without any assumption on the existence of resolutions of singularities, it is not known whether $\wt_{\form}$ is bounded below and $\Sk(X,\form)$ is non-empty.
 However, Theorem 4.7.5 in \cite{MuNi} provides an explicit description of $\Sk(X,\form)$ if we assume that $X$ has an $snc$-model $\cX$ over $R$.
 In that case, we can compute $\wt_{\min}(X,\form)$ by taking the minimum of the weights $\wt_{\form}(x)$ where $x$ runs through the finite set of divisorial points associated with the prime components of $\cX_k$. We write $$\cX_k=\sum_{i\in I}N_i E_i, \quad \mathrm{div}_{\cX}(\form)=\sum_{i\in I}w_i E_i + H$$
 with $H$ an effective Cartier divisor that is horizontal, that is, whose support does not contain any of the components $E_i$ of $\cX_k$. Then we have
 $$\wt_{\min}(X,\form)=\min \{\frac{w_i}{mN_i}\,|\,i\in I\}.$$
    Moreover, we can identify $\Sk(X,\form)$ with the union of the so-called $\form$-{\em essential} faces of the Berkovich skeleton $\Sk(\cX)$ of $\cX$.
  This Berkovich skeleton is a subspace of $X^{\an}$ that is canonically homeomorphic to the dual intersection complex of $\cX_k$, and a face $\sigma$ is called $\form$-essential if it satisfies the following two conditions:
  \begin{itemize}
  \item for every vertex $v$ of $\sigma$, we have $\wt_{\form}(v)=\wt_{\min}(X,\form)$ (recall that the vertices of $\Sk(\cX)$ are precisely the divisorial points associated with the prime components of $\cX_k$);

  \item the stratum of $\cX_k$ corresponding to $\sigma$ is not contained in the horizontal part $H$ of the Cartier divisor $\mathrm{div}_{\cX}(\form)$.
  \end{itemize}
  In particular, $\Sk(X,\form)$ is a non-empty compact subspace of $X^{\an}$. Likewise, $\Sk(X)$ is a compact subspace of $X^{\an}$, and it is non-empty if and only if $X$ has nonnegative Kodaira dimension.
  It is proven in~\cite[3.3.5]{NiXu} and~\cite[Thm.24]{KNX} that $\Sk(X)$ is a strong deformation retract of $X^{\an}$ if $k$ has characteristic zero, $X$ is projective, and the canonical line bundle $\omega_{X/K}$ of $X$ is semi-ample.

\subsection{Weak N\'eron models and the Shilov boundary}\label{sec:shilov}
 Let $X$ be a smooth and proper $K$-scheme of pure dimension.

 \begin{definition}
 A {\em weak N\'eron model} of $X$ is a smooth $R$-model $\cU$ of $X$ such that the map $\cU(R^{\ur})\to X(K^{\ur})$ is bijective.
  \end{definition}

  The $R$-scheme $\cU$ is not necessarily proper, but the definition expresses that it satisfies the valuative criterion for properness at least for unramified extensions of $R$. It is proven in~\cite{BLR} that a weak N\'eron model always exists, without any restriction on the characteristic of $k$; however, it is far from unique. Given any proper $R$-model $\cX$ of $X$, we can always find a morphism $\cY\to \cX$ of $R$-models that is a composition of blow-ups at centers supported in the special fiber and such that every $R^{\ur}$-valued point on $\cY$ factors through the $R$-smooth locus $\Sm(\cY)$ of $\cY$, by~\cite[3.4/2]{BLR}. Then $\Sm(\cY)$ is a weak N\'eron model of $X$.  Such a morphism $\cY\to \cX$ is called a {\em N\'eron smoothening} of $\cX$.

   Let $\form$ be an $m$-canonical form on $X$, for some positive integer $m$, and assume that $\form$ is not identically zero on any connected component of $X$.
  For every weak N\'eron model $\cU$ of $X$ and every connected component $C$ of $\cU_k$, we set
  $$\ord_C(\form)=\frac{w}{m}$$ where $w$ is the unique integer such that $\unif^{-w} \form$ extends to a generator of $\omega^{\otimes m}_{\cU/R}$ at the generic point of $C$.
   If we denote by $x$ the divisorial point on $X^{\an}$ associated with $C$, then $\ord_C(\form)=\wt_{\form}(x)$, because $C$ has multiplicity $1$ in $\cU_k$.

\begin{definition}\label{defi:shilov}
Let $\cU$ be a weak N\'eron model of $X$. We set $$\ordmin(\cU,\form)=\min\{\mathrm{ord}_C(\form)\,|\,C\in \pi_0(\cU_k)\}.$$
The {\em Shilov boundary} of $(\cU,\form)$ is the set of divisorial points in $X^{\an}$ associated with the connected components $C$ of $\cU_k$ for which $\mathrm{ord}_C(\form)=\ordmin(\cU,\form)$. We denote this set by $\Sh(\cU,\form)$.
\end{definition}
Note that $\Sh(\cU,\form)$ is empty if $X(K^{\ur})$ is empty; in that case, every weak N\'eron model $\cU$ of $X$ has empty special fiber, and $\ordmin(\cU,\form)=-\infty$.

\begin{prop}\label{prop:indep}
The value $\ordmin(\cU,\form)$ and the subset $\Sh(\cU,\form)$ of $X^{\an}$ only depend on the pair $(X,\form)$, and not on the choice of a weak N\'eron model $\cU$.
\end{prop}
\begin{proof}
Let $\cU$ and $\cU'$ be weak N\'eron models for $X$. Then we can dominate both of them by a third weak N\'eron model $\cV$ by taking
 the schematic closure of the diagonal embedding of $X$ in $\cU\times_R \cU'$, and applying a N\'eron smoothening. We denote the morphism $\cV\to \cU$ by $h$.
     By Zariski's Main Theorem, $h$ is quasi-finite locally around a point $\xi$ of $\cV_k$ if and only if it is an open immersion locally around $\xi$, because $h_K$ is an isomorphism, $\cU$ is normal, and $\cV$ is $R$-flat. Thus if $C$ is a connected component of $\cV_k$, then either $h$ is an open immersion around the generic point of $C$, or $h(C)$ does not contain any generic point of $\cU_k$.

  For every generic point $\xi$ of $\cV_k$, we have that $(h^*\omega^{\otimes m}_{\cU/R})_{\xi}$ is a submodule of $(\omega^{\otimes m}_{\cV/R})_{\xi}$, and they are equal if and only if
  $h$ is \'etale at $\xi$.
  Thus if $C$ is a connected component of $\cV_k$ and $C'$ is the unique connected component of $\cU_k$ that contains $h(C)$, then
   $\ord_{C}\form\geq \ord_{C'}\form$, and equality holds if and only if $h$ is an open immersion locally around the generic point of $C$.
    We also know that $h_k:\cV_k\to \cU_k$ is dominant, because every $k^s$-rational point on $\cU_k$ can be lifted to an element of $\cU(R^{\ur})=\cV(R^{\ur})$ by the infinitesimal lifting criterion of smoothness, and $k^s$-rational points are dense in $\cU_k$ because $\cU_k$ is smooth over $k$.
     Therefore, every generic point of $\cU_k$ lifts to a generic point of $\cV_k$.
    It follows that $\ordmin(\cU,\form)=\ordmin(\cV,\form)$
    and $\Sh(\cU,\form)=\Sh(\cV,\form)$.
\end{proof}

In view of Proposition~\ref{prop:indep}, we will henceforth write $\ordmin(X,\form)$ and $\Sh(X,\form)$ instead of  $\ordmin(\cU,\form)$ and $\Sh(\cU,\form)$, and we call
$\Sh(X,\form)\subset X^{\an}$ the Shilov boundary of the pair $(X,\form)$.

\section{Logarithmic models}\label{sec:log}
We assume that the reader has a basic familiarity with logarithmic geometry. For a general introduction to log schemes, we refer to~\cite{kato-log}. The theory of regular log schemes is due to Kato~\cite{kato}; the properties and constructions that will be used below are summarized in Section 3 of~\cite{logzeta}.  All log structures in this paper are defined with respect to the Zariski topology.
\subsection{Log regular and log smooth models}
Let $\cX$ be a separated flat $R$-scheme of finite type. We set $S=\Spec R$. The {\em standard log structure} on $S$ is the divisorial log structure induced by the unique closed point of $S$. Likewise, the standard log structure on $\cX$ is the divisorial log structure induced by the special fiber $\cX_k$.

We say that $\cX$ is {\em log regular} if the scheme $\cX$ endowed with its standard log structure
 is log regular in the sense of~\cite{kato}. Intuitively, this property expresses that the pair $(\cX,\cX_{k,\red})$ is toroidal, but the notion of log regularity is more flexible (in particular, it also applies to the case where $R$ has mixed characteristic).  A {\em stratum} of $\cX_k$ is a connected component of an intersection of irreducible components of $\cX_k$.
 We say that $\cX$ is {\em log smooth} over $S$ if the morphism $\cX\to S$ is log smooth if we endow $\cX$ and
 $S$ with their standard log structures.
  If $\cX$ is log smooth over $S$ then it is log regular~\cite[8.2]{kato}; if $k$ has characteristic zero, then both notions are equivalent~\cite[3.6.1]{logzeta}.
 We will repeatedly make use of the fact that log regular schemes are normal~\cite[4.1]{kato}.

\begin{example}
Let $X$ be a smooth proper $K$-scheme. Every $snc$-model $\cX$ of $X$ is log regular. If, moreover, the multiplicities of the components of $\cX_k$ are prime to $p$ and the residue field $k$ is perfect, then $\cX$ is also log smooth over $S$ (but this is not a necessary condition for log smoothness). In particular, every $snc$-model is log smooth if $k$ has characteristic zero.
\end{example}

   If a smooth proper $K$-scheme $X$ has a log regular proper $R$-model, then we can also find an $snc$-model for $X$, by toroidal resolution of singularities for log regular schemes~\cite{kato}.
    Thus the existence of log regular proper models is known in the same cases where $snc$-models are known to exist: when $k$ has characteristic zero, and when $k$ has positive characteristic and the dimension of $X$ is at most $2$.
  If $X$ has a log smooth proper $R$-model, then the wild inertia acts trivially on the $\ell$-adic cohomology of $X$~\cite[0.1.1]{nakayama-nearby}.
 The converse implication holds if $X$ is a curve of genus at least $2$ or an abelian variety of arbitrary dimension~\cite{saito-sstable, stix, bellardini-smeets}. No similar criterion is known in general, but it is expected that there always exists a finite extension $K'$ of $K$ such that $X\otimes_K K'$ has a log smooth proper $R$-model. This is equivalent to the existence of a {\em semistable} $snc$-model, that is, an $snc$-model with reduced special fiber, over some finite extension of $R$, by~\cite{saito-sstable}.

 \subsection{Behavior under normalized base change}
 \begin{prop}\label{prop:logreg}
  Let $\cX$ be a log regular separated flat $R$-scheme of finite type.
  Let $K'$ be a finite extension of $K$, and let $R'$ be the integral closure of $R$ in $K'$. Denote by $\cX'$ the normalization of $\cX\otimes_R R'$.
  \begin{itemize}
  \item[(i)]
    If $\cX$ is log smooth over $S$, then $\cX'$ is log smooth over $S'=\Spec(R')$ (in particular, it is log regular).
  \item[(ii)]
    If $K'$ is tamely ramified over $K$ and $\cX$ is log regular, then the $R'$-scheme $\cX'$ is log regular.
  \end{itemize}
\end{prop}
\begin{proof}
  We set $S'=\Spec(R')$ and we  endow $S$, $S'$ and $\cX$ with their standard log structures. The condition that $K'$ is tamely ramified over $K$ is equivalent to the property that $S'$ is log \'etale over $S$.
  Log smoothness is preserved by base change in the category of fine and saturated (fs) log schemes; thus the fs base change $\cY$ of $\cX$ to $S'$ is log smooth over $S'$ (if $\cX$ is log smooth over $S$) or over $\cX$ (if $K'$ is tamely ramified over $K$). In both cases, $\cY$ is log regular, since it is log smooth over a log regular scheme.
   Now it follows from the proof of Proposition 3.7.1 in~\cite{logzeta} that the underlying scheme of $\cY$ is the normalization of $\cX\otimes_R R'$.
  By~\cite[11.6]{kato}, log regularity also implies that the locus in $\cY$ where the log structure is not trivial is a divisor $D$, and that the log structure on $\cY$ is the divisorial log structure with respect to $D$. In our case, $D=\cY_k$, so that $\cY=\cX'$ with the standard log structure.
\end{proof}

The following result describes how the multiplicities of the components in the special fiber of a log regular model behave under normalized base change.
\begin{prop}\label{prop:saturated}
Let $\cX$ be a log regular separated flat $R$-scheme of finite type.
 Let $K'$ be a finite extension of $K$ of ramification degree $e$. Assume either that $\cX$ is log smooth over $S$, or that $K'$ is a tame extension of $K$.
 Denote by $R'$ the integral closure of $R$ in $K'$, and by $\cX'$ the normalization of $\cX\otimes_R R'$.

 Let $E$ be an irreducible component of $\cX_k$, and denote its multiplicity in $\cX_k$ by $N$.
 Then the multiplicity of the special fiber of $\cX'$ along each component dominating $E$ is equal to $N/\gcd(e,N)$.
\end{prop}
\begin{proof}
 We set $S'=\Spec R'$.
We have seen in the proof of Proposition~\ref{prop:logreg} that the normalization morphism $\cX'\to \cX\otimes_R R'$ coincides with the saturation morphism of
the fibered product $\cX\times_S S'$ in the category of log schemes, where $S$, $S'$ and $\cX$ are endowed with their standard log structures. We can determine the multiplicities of the components in the special fiber of the saturation of $\cX\otimes_R R'$ by means of a local computation on characteristic monoids.

For every log scheme $\cY$, we denote by $\mathcal{C}_{\cY}=\mathcal{M}_{\cY}/\mathcal{O}^{\times}_{\cY}$ its sheaf of characteristic monoids. At every point $\zeta$ of $\cX_k$, the morphism of log schemes
$\cX\to S$ gives rise to a structural morphism of characteristic monoids $\mathcal{C}_{S,s}=\N\to \mathcal{C}_{\cX,\zeta}$, where $s$ denotes the closed point of $S$. The analogous observation applies to log schemes over $S'$. The morphism of characteristic monoids associated with $S'\to S$ is the inclusion map $\N\to (1/e)\N$.

 Now let $\zeta$ be the generic point of $E$, and let $\xi$ be any point of $\cX\otimes_R R'$ lying above $\zeta$.
 Then the characteristic monoid
 $\mathcal{C}_{\cX\otimes_R R',\xi}$ is given by the coproduct
 $$\mathcal{C}_{\cX,\zeta}\oplus_{\N}(1/e)\N$$ modulo its submonoid of units (the proof of~\cite[2.1.1]{nakayama} remains valid for log structures on the Zariski site).
  The monoid $\mathcal{C}_{\cX,\zeta}$ is equal to $(1/N)\N$, and its structural morphism is the inclusion map $\N\to (1/N)\N$. The saturation of the monoid $(1/N)\N\oplus_{\N}(1/e)\N$ is isomorphic to $(1/\lcm(e,N))\N\oplus (\Z/\gcd(e,N)\Z)$, so that the characteristic monoid at any point $\xi'$ of
$\cX'$ lying above $\xi$ is given by $(1/\lcm(e,N))\N$, with structural morphism $$(1/e)\N\to (1/\lcm(e,N))\N\colon 1/e\mapsto 1/e.$$ This means that the multiplicity of the special fiber of $\cX'$ at $\xi'$ is equal to $\lcm(e,N)/e=N/\gcd(e,N)$.
\end{proof}
\begin{cor}\label{coro:ppower}
  Let $\cX$ be a log regular separated flat $R$-scheme of finite type.
  Then there exists a tame finite extension $K'$ of $K$, with valuation ring $R'$, such that
  the normalization $\cX'$ of $\cX\otimes_R R'$ has the following properties: the multiplicity of $\cX'_k$ along each of its
 irreducible components is a power of $p$, and every stratum of $\cX'_k$ is geometrically connected.
\end{cor}
\begin{proof}
Let $e$ be the least common multiple of the multiplicities of the irreducible components in $\cX_k$, and let $e'$ be its prime-to-$p$ part. Let $L$ be a totally ramified extension of $K$ of degree $e'$, with valuation ring $R_L$. Then the normalization $\cY$ of $\cX\otimes_R R_L$ is log regular and the multiplicities of the components of its special fiber are powers of $p$, by Proposition~\ref{prop:saturated}. Let $k'$ be a finite separable extension of $k$ such that the connected components of $U\times_k k'$ are geometrically connected for every stratum $U$ of $\cY_k$, and let $K'$ be the unique unramified extension of $L$ with residue field $k'$. We denote by $R'$ the valuation ring of $K'$. The scheme $\cY\otimes_{R_L} R'$ is normal, because it is \'etale over the normal scheme $\cY$. Thus $\cY\otimes_{R_L} R'$ is canonically isomorphic to $\cX'$. By our choice of $K'$, the multiplicity of $\cX'_k$ along each of its irreducible components is a power of $p$, and every stratum of $\cX'_k$ is geometrically connected.
\end{proof}
The same argument also gives
\begin{cor}\label{coro:reduced}
  Let $\cX$ be a log smooth separated flat $R$-scheme of finite type.
  Then there exists a finite extension $K'$ of $K$, with valuation ring $R'$, such that
  the normalization $\cX'$ of $\cX\otimes_R R'$ has the following properties: $\cX'_k$ is reduced and every stratum of $\cX'_k$ is geometrically connected.
\end{cor}
\begin{proof}
We can simply copy the proof of Corollary \ref{coro:ppower}, replacing $e'$ by $e$.
\end{proof}

\subsection{Logarithmic differential forms}\label{ss:logdiff}
 Let $\cX$ be a log regular separated flat $R$-scheme of finite type.
We denote by $\Omega^{\mathrm{log}}_{\cX/R}$ the sheaf of logarithmic differential forms on $\cX$ relative to $S=\Spec(R)$, where both $\cX$ and $S$ are equipped with their standard log structures. Since the log structure is trivial on $\cX\otimes_R K$, the restriction of $\Omega^{\mathrm{log}}_{\cX/R}$ to $\cX\otimes_R K$ is canonically isomorphic to the usual sheaf of differentials $\Omega_{\cX\otimes_R K/K}$.

 If $\cX$ is log smooth over $S$, then $\Omega^{\mathrm{log}}_{\cX/R}$ is locally free, and we denote its determinant line bundle by
 $\omega_{\cX/R}^{\mathrm{log}}$.
 If $\cX$ is only log regular, we define $\omega_{\cX/R}^{\mathrm{log}}$ in the following way. Denote by $\cX_{\mathrm{reg}}$ the regular locus of $\cX$; this is a dense open subscheme of $\cX$, and its complement has codimension at least two, because $\cX$ is normal. Since coherent sheaves on regular schemes are perfect, we can define the determinant of $\Omega^{\mathrm{log}}_{\cX_{\mathrm{reg}}/R}$ as in~\cite{knudsen-mumford}, by locally taking a finite resolution by free coherent sheaves of finite rank, and taking the alternating product of their determinants. The result is a line bundle $\omega_{\cX_{\mathrm{reg}}/R}^{\mathrm{log}}$ on $\cX_{\mathrm{reg}}$, and we define $\omega^{\mathrm{log}}_{\cX/R}$ to be the pushforward of this line bundle to $\cX$. This is a reflexive sheaf on $\cX$. On the other hand, we can also forget the log structures and consider the usual relative canonical sheaf $\omega_{\cX/R}$ (the pushforward of the determinant of
 $\Omega_{\cX_{\mathrm{reg}}/R}$). The following proposition collects the basic properties of logarithmic canonical sheaves that we will need; we include detailed proofs because we have not been able to find them in the literature.

 \begin{prop}\label{prop:can}
 Assume that $k$ is perfect.
 Let $\cX$ be a log regular separated flat $R$-scheme of finite type.
\begin{enumerate}
\item \label{it:logcan} Denote by $i\colon \cX\otimes_R K\to \cX$ the inclusion map. We have
$$\omega^{\log}_{\cX/R}=\omega_{\cX/R}(\cX_{k,\red}-\cX_k)$$
as submodules of $i_\ast \omega_{\cX\otimes_R K/K}$.
\item \label{it:lineb} The coherent sheaf $\omega^{\log}_{\cX/R}$ is a line bundle on $\cX$.
\item \label{it:loget}
Let $h\colon \cY\to \cX$ be a morphism of log schemes associated with a (not necessarily proper) subdivision of the fan of $\cX$~\cite[9.9]{kato}. Then
$$\omega^{\log}_{\cY/R}=h^{\ast}\omega^{\mathrm{log}}_{\cX/R}$$
as submodules of $j_\ast \omega_{\cY\otimes_R K/K}$, where $j$ denotes the inclusion map $$\cY\otimes_R K\to \cY.$$
\item \label{it:canbc} Let $K'$ be a finite extension of $K$, and let $R'$ be the integral closure of $R$ in $K'$.
 Set $S'=\Spec(R')$ and let $\cX'$ be the normalization of $\cX\otimes_R R'$. We denote by $g\colon \cX'\to \cX$ the projection morphism.
Assume either that $\cX$ is log smooth over $S$ or that $K'$ is tamely ramified over $K$.
 Then $$\omega^{\mathrm{log}}_{\cX'/R'}=g^{\ast}\omega^{\mathrm{log}}_{\cX/R}$$
 as submodules of $(i\otimes_K K')_\ast \omega_{\cX'\otimes_{R'} K'/K'}$.
\end{enumerate}
\end{prop}
\begin{proof}
\eqref{it:logcan} Since $\cX$ is normal and both sheaves are reflexive, it suffices to prove the equality on the maximal open subscheme $\cU$ of $\cX$ such that $\cU$ is regular and $\cU_k$ is a divisor with strict normal crossings (note that the complement of $\cU$ in $\cX$ has codimension at least two, since $\cU$ contains the generic fiber $\cX\otimes_R K$ and the generic points of $\cX_k$). On the open $\cU$, the equality follows from~\cite[5.3.4]{kato-saito} by taking determinants.

\eqref{it:lineb}
Let $S=\Spec(R)$.
 To avoid ambiguities, we denote by $S_{\mathrm{tr}}$ and $S_{\mathrm{st}}$ the scheme $S=\Spec(R)$ endowed with its {\em trivial}, resp.~{\em standard}, log structure.
   The result in~\cite[5.3.4]{kato-saito} implies that, on the open $\cU$, the determinant of the sheaf of log differentials $\Omega^{\mathrm{log}}_{\cU/S_{\mathrm{tr}}}$ is isomorphic to the determinant of the sheaf $\Omega^{\mathrm{log}}_{\cU/S_{\mathrm{st}}}$, where $\cU$ carries the standard log structure in both cases. Thus it suffices to prove that the former determinant extends to a line bundle on $\cX$.

 This property can be checked locally around closed points $\xi$ of $\cX_k$;
  it is enough to prove that the $\mathcal{O}_{\cX,\xi}$-module $\Omega^{\mathrm{log}}_{\cX/S_{\mathrm{tr}},\xi}$ has a finite resolution by free modules of finite rank. Our proof is similar to that of Corollary 5.3.2 in~\cite{kato-saito}.
  We fix a section of the projection morphism $$\mathcal{M}_{\cX,\xi}\to \mathcal{C}_{\cX,\xi}=\mathcal{M}_{\cX,\xi}/\mathcal{O}^{\times}_{\cX,\xi}$$ so that we can view the characteristic monoid $\mathcal{C}_{\cX,\xi}$ as a submonoid of the multiplicative monoid $\mathcal{O}_{\cX,\xi}$. Such a section exists because $\mathcal{C}_{\cX,\xi}$ is integral and its groupification $\mathcal{C}^{\mathrm{gp}}_{\cX,\xi}$ is a free $\Z$-module. We can extend this section to a surjective morphism of $R$-algebras
  $$R[\mathcal{C}_{\cX,\xi}][T_1,\ldots,T_r]\to \mathcal{O}(\cV)$$ where $\cV$ is an affine open neighborhood of $\xi$ in $\cX$.
   We set $$\cW=\Spec R[\mathcal{C}_{\cX,\xi}][T_1,\ldots,T_r]$$ and we endow it with the log structure induced by the monoid $\mathcal{C}_{\cX,\xi}$.
   Then $\cW$ is log smooth over $S_{\mathrm{tr}}$. Shrinking $\cV$ around $\xi$, we can arrange that $\cV$ is a strict logarithmic subscheme of $\cW$.
   Since $\cW$ and $\cV$ are both log regular, the immersion $\cV\to \cW$ is regular, by~\cite[4.2]{kato}. If we denote by $J$ the defining ideal of $\cV$ in $\mathcal{O}_{\cW}$ at the point $\xi$, then we have a short exact sequence
\begin{equation}\label{eq:res}
0\to J/J^2 \to \Omega^{\mathrm{log}}_{\cW/S_{\mathrm{tr}},\xi}\otimes_{\mathcal{O}_{\cW,\xi} } \mathcal{O}_{\cX,\xi}\to \Omega^{\mathrm{log}}_{\cX/S_{\mathrm{tr}},\xi} \to 0.
\end{equation}
Since $\cW$ is log smooth over $S_{\mathrm{tr}}$ and $J$ is generated by a regular sequence, this short exact sequence is a resolution of $\Omega^{\mathrm{log}}_{\cX/S_{\mathrm{tr}},\xi}$ by free modules of finite rank.

\eqref{it:loget} We continue to use the notation from the proof of~\eqref{it:lineb}. By~\cite[5.3.4]{kato-saito}, we have
$$\omega^{\mathrm{log}}_{\cX/S_{\mathrm{st}}}\cong \det(\Omega^{\mathrm{log}}_{\cX,S_{\mathrm{tr}}})\otimes \mathcal{O}_{\cX}(\cX_k)$$ as subsheaves of
 $i_\ast \omega_{\cX\otimes_R K/K}$, and the analogous statement holds on $\cY$. Thus it suffices to prove the result for $\det(\Omega^{\mathrm{log}}_{\cX,S_{\mathrm{tr}}})$.
 We have a natural isomorphism $\Omega^{\mathrm{log}}_{\cY/S_{\mathrm{tr}}}\cong h^{\ast}\Omega^{\mathrm{log}}_{\cX/S_{\mathrm{tr}}}$ because $h$ is log \'etale.
  Taking determinants of perfect complexes does not commute with non-flat base change, in general, but here we can make use of the local free resolution in~\eqref{eq:res}.
  The pullback through $h$ of~\eqref{eq:res} is still exact, because $h$ is flat over $\cX_K$ (even an isomorphism) and $\cY$ is $R$-flat.
 Taking determinants, we get that $$\det(\Omega^{\mathrm{log}}_{\cY/S_{\mathrm{tr}}})=h^{\ast}\det(\Omega^{\mathrm{log}}_{\cX/S_{\mathrm{tr}}}).$$

\eqref{it:canbc} We have seen in the proof of Proposition~\ref{prop:logreg} that $\cX'$ with its standard log structure is the
 base change of $\cX$ to $S'$ in the category of fine and saturated log schemes. Since log differentials are compatible with fine and saturated base change, the natural morphism
$$g^{\ast} \Omega^{\mathrm{log}}_{\cX/R}\to \Omega^{\mathrm{log}}_{\cX'/R'}$$ is an isomorphism.
 The scheme $\cX'$ is normal and the sheaves we want to compare are reflexive (even line bundles, by~\eqref{it:lineb}), so that it suffices to check the equality at the generic points of $\cX'_k$. At these points, the morphism $g$ is flat, so that the result follows from the compatibility of determinants with flat base change.
\end{proof}

\section{Skeleta of log regular models}\label{sec:skeleta}
Let $X$ be a smooth and proper $K$-scheme. In this section we study the skeleton $\Sk(\cX)\subset X^{\an}$ associated to a log regular $R$-model $\cX$ of $X$,
and the Kontsevich--Soibelman skeleton $\Sk(X,\theta)$ associated to a pluricanonical form on $X$. In particular, we analyze the behavior of these skeleta under finite ground field extension.
\subsection{Skeleta of log regular models}\label{ss:logskeleta}
  Let $\cX$ be a log regular $R$-model of $X$ (where $\cX$ carries the standard log structure). It is explained in ~\cite[\S3]{BrMa} how one can attach to this model a skeleton $\Sk(\cX)$ in $X^{\an}$, generalizing the construction of the skeleton of an $snc$-model in~\cite[\S3]{MuNi}. Let us briefly summarize the construction.
 Recall that a {\em stratum} of $\cX_k$ is a connected component of an intersection of irreducible components of $\cX_k$.
The {\em Kato fan} of $\cX$
is the subspace $F(\cX)$ of $\cX$ consisting of the generic points of $\cX$ and the generic points of all the strata of $\cX_k$, endowed with the restriction of the sheaf of characteristic  monoids $\mathcal{C}_{\cX}=\mathcal{M}_{\cX}/\mathcal{O}^{\times}_{\cX}$ of the log scheme $\cX$. We will denote this restriction by $\mathcal{C}_{F(\cX)}$. The stalks of this sheaf are
 sharp finitely generated saturated integral monoids; equivalently, monoids of integral points in strictly convex  rational polyhedral cones in $\R^n$.
 Since $\mathcal{M}_{\cX}$ is the sheaf of regular functions on $\cX$ that are invertible on $X$,
 we can interpret the stalk $\mathcal{C}_{\cX,\xi}$ at a point $\xi$ of $F(\cX)$ as the monoid of effective Cartier divisors on $\Spec \mathcal{O}_{\cX,\xi}$ that are supported on the special fiber $(\Spec \mathcal{O}_{\cX,\xi})_k$.
   The morphism $\cX\to S$ induces for every point $\xi$ of $F(\cX)\cap \cX_k$ a morphism of characteristic monoids $\N\to \mathcal{C}_{F(\cX),\xi}$ that maps $1$ to the residue class $\overline{\unif}$ of the uniformizer $\unif$ in $\mathcal{C}_{F(\cX),\xi}$ (equivalently, we can think of $\overline{\unif}$ as the Cartier divisor $(\Spec \mathcal{O}_{\cX,\xi})_k$ on $\Spec \mathcal{O}_{\cX,\xi}$). We call this morphism the structural morphism of the monoid $\mathcal{C}_{F(\cX),\xi}$.

 When $\xi$ and $\eta$ are points in $F(\cX)$ such that $\xi$ lies in the closure of $\eta$, the cospecialization map $\mathcal{C}_{F(\cX),\xi}\to \mathcal{C}_{F(\cX),\eta}$
  induces a morphism of dual real cones $$\Hom(\mathcal{C}_{F(\cX),\eta},\R_{\geq 0})\to \Hom(\mathcal{C}_{F(\cX),\xi},\R_{\geq 0}) $$
 that identifies $\Hom(\mathcal{C}_{F(\cX),\eta},\R_{\geq 0})$ with a face of $\Hom(\mathcal{C}_{F(\cX),\xi},\R_{\geq 0})$. Gluing the dual real cones according to these face maps, we
 get a cone complex with integral affine structure~\cite{KKMS}.
 The skeleton $\Sk(\cX)$ is the polyhedral complex
  consisting of the elements $\alpha$ in the dual cones $\Hom(\mathcal{C}_{F(\cX),\xi},\R_{\geq 0})$ such that
  $\alpha(\overline{\unif})=1$ (note that this set is empty if $\xi$ is not contained in $\cX_k$, since $\overline{\unif}=0$ in that case). The faces of $\Sk(\cX)$ carry a canonical integral $K$-affine structure; the integral $K$-affine functions on the face
  $$\{\alpha\in \Hom(\mathcal{C}_{F(\cX),\xi},\R_{\geq 0})\,|\,\alpha(\overline{\unif})=1\}$$
  are the functions of the form
$$\alpha\mapsto \alpha^{\gp}(u)+ v$$
where $u\in \mathcal{C}^{\gp}_{F(\cX),\xi}$ and $v\in \log_{\varepsilon}|K^{\times}|=\Z$, and where $\alpha^{\gp}$ is the group homomorphism $\mathcal{C}^{\gp}_{F(\cX),\xi}\to \R$ induced by $\alpha$.
  For every integer $d\geq 0$, we denote by $\Sk^{(d)}(\cX)$ the $d$-skeleton of $\Sk(\cX)$, that is, the union of the faces of dimension at most $d$.

\begin{example}\label{exam:snc}
If $\cX$ is regular and $\cX_k$ has strict normal crossings, then the characteristic monoid $\mathcal{C}_{F(\cX),\xi}$ at a point $\xi$ of $F(\cX)$ is isomorphic to $\N^J$, where $J$ is the set of irreducible components of $\cX_k$ passing through $\xi$. If we denote by $N_j$ the multiplicity of the $j$-th component, for every $j$ in $J$, then the condition
$\alpha(\overline{\unif})=1$ on the dual cone $\Hom(\mathcal{C}_{F(\cX),\xi},\R_{\geq 0})=\R^J_{\geq 0}$ defines the simplex
$$\{u\in \R^J_{\geq 0}\,|\,\sum_{j\in J}N_ju_j=1\}.$$ In this case, the construction of $\Sk(\cX)$ is equivalent to the one in~\cite[\S3]{MuNi}.
\end{example}

There is a canonical embedding of $\Sk(\cX)$ into $\widehat{\cX}_{\eta}$, and a canonical continuous retraction
$$\rho_{\cX}\colon \widehat{\cX}_{\eta}\to \Sk(\cX).$$
The skeleton $\Sk(\cX)$ is made up of points on $X^{\an}$ corresponding to valuations that are {\em monomial} with respect to the model $\cX$; these valuations are constructed explicitly in~\cite[3.2.10]{BrMa}.
 The following proposition provides a convenient criterion for such points, as well as a characterization of the retraction map $\rho_{\cX}$. We recall that, since the log structure on $\cX$ is the divisorial log structure induced by the special fiber, the monoid $\mathcal{M}_{\cX,\xi}$ is the multiplicative monoid of elements in $\mathcal{O}_{\cX,\xi}$ that are invertible in $\mathcal{O}_{\cX,\xi}\otimes_R K$, for every point $\xi$ in $\cX_k$.

\begin{prop}\label{prop:maxmon}
Let $\cX$ be a log regular $R$-model of $X$. Then a point $x$ of $\widehat{\cX}_{\eta}$ lies in the skeleton $\Sk(\cX)$ if and only if
$|f(x)|\geq |f(y)|$ for every $f$ in $\mathcal{O}_{\cX,\spe_{\cX}(x)}$ and for every $y$ in $\widehat{\cX}_{\eta}$ satisfying the following properties:
\begin{enumerate}
\item the point $\spe_{\cX}(x)$ lies in the closure of $\{\spe_{\cX}(y)\}$;
\item we have $|g(x)|=|g(y)|$ for every $g$ in $\mathcal{M}_{\cX,\spe_{\cX}(x)}$.
\end{enumerate}
If $x$ lies in $\Sk(\cX)$ and we set $\xi=\spe_{\cX}(x)$, then the value group $|\mathscr{H}(x)^{\times}|$ is the submonoid of $(\R_{>0},\cdot)$ consisting of the elements $|f(x)|$ with $f$ in $\mathcal{M}_{\cX,\xi}$, and the residue field of $\mathscr{H}(x)$ is a purely transcendental extension of the residue field $k(\xi)$ of $\cX$ at $\xi$.

Moreover, for every point $z$ of $\widehat{\cX}_{\eta}$, the point $\rho_{\cX}(z)$ is the unique point $w$ in $\Sk(\cX)$ such that $\spe_{\cX}(z)$ lies in the closure of $\{\spe_{\cX}(w)\}$ and $|g(z)|=|g(w)|$ for every $g$ in the monoid $\mathcal{M}_{\cX,\spe_{\cX}(z)}$.
\end{prop}
\begin{proof}
The criterion for $x$ to lie in $\Sk(\cX)$ is a reformulation of Proposition 3.2.10 in~\cite{BrMa} (the proof remains valid if the valuation $w$ considered there is replaced by a semivaluation). The assertions about the value group and residue field of $\mathscr{H}(x)$ follow immediately from the definition of monomial valuations.
 The description of the retraction map $\rho_{\cX}$ follows from the definition of $\rho_{\cX}$ in \S3.3 of~\cite{BrMa}.
\end{proof}

 \begin{prop}\label{prop:compat1}
  Let $\cX$ and $\cY$ be log regular proper $R$-models of $X$ such that $\cY$ dominates $\cX$, {\em i.e.}, there exists a morphism of $R$-schemes $\cY\to \cX$ that restricts to the identity on the generic fiber $X$. Then $\Sk(\cX)\subset\Sk(\cY)$, and for every integer $d\geq 0$, the $d$-skeleton of $\Sk(\cX)$ is contained in the $d$-skeleton of $\Sk(\cY)$.
 \end{prop}
\begin{proof}
  We first prove the inclusion $\Sk(\cX)\subset \Sk(\cY)$.
  Let $x$ be a point in $\Sk(\cX)$ and
 set $y=\rho_{\cY}(x)$; then we need to show that $x=y$. It is enough to prove that $\rho_{\cX}(x)=y$, because $x$ lies in $\Sk(\cX)$ so that $\rho_{\cX}(x)=x$.
  By the definition of the retraction map $\rho_{\cY}$, we know that $\spe_{\cY}(x)$ lies in the closure of $\{\spe_{\cY}(y)\}$
  and that $|g(x)|=|g(y)|$ for every element $g$ in $\mathcal{M}_{\cY,\spe_{\cY}(x)}$.
  Thus $\spe_{\cX}(x)$ lies in the closure of $\{\spe_{\cX}(y)\}$, because the specialization maps commute with the morphism $\cY\to \cX$, and
   $|g(x)|=|g(y)|$ for every element $g$ in $\mathcal{M}_{\cX,\spe_{\cX}(x)}\subset \mathcal{M}_{\cY,\spe_{\cY}(x)}$.
   By the maximality property in Proposition~\ref{prop:maxmon}, it suffices to show that $|f(x)|\leq |f(y)|$ for every $f$ in $\mathcal{O}_{\cX,\spe_{\cX}(x)}$.
  Since $f$ also lies in $\mathcal{O}_{\cY,\spe_{\cY}(x)}$, this inequality follows from Proposition~\ref{prop:maxmon}, applied to the model $\cY$.

 Now, we prove the inclusion of $d$-skeleta.
  The points in the $d$-skeleton of $\Sk(\cX)$ are precisely the points $x$ in $\Sk(\cX)$ such that the closure of $\spe_{\cX}(x)$ has codimension at most $d$ in $\cX_k$; the analogous property holds for $\cY$.
  Since the morphism $\cY\to \cX$ commutes with the specialization maps and $\Sk(\cX)\subset \Sk(\cY)$, the $d$-skeleton of $\Sk(\cX)$ is contained in that of $\Sk(\cY)$.
\end{proof}

\subsection{Behavior under ground field extension}
Now consider a finite extension $K'$ of $K$. Set $X':=X\otimes_KK'$, and let
$R'$ be the integral closure of $R$ in $K'$. We shall study the behavior of skeleta under
the canonical surjective map
\begin{equation*}
  \mathrm{pr}\colon (X')^{\an}\to X^{\an}.
\end{equation*}
\begin{prop}\label{prop:compat2}
Let $\cX$ be a log regular proper $R$-model of $X$, and let $\cX'$ be the normalization of $\cX\otimes_R R'$. Assume that $\cX$ is log smooth over $R$, or that $K'$ is tamely ramified over $K$.
Then $\cX'$ is a log regular proper $R'$-model of $X'$,
$\mathrm{pr}^{-1}(\Sk(\cX))=\Sk(\cX')$,
and the restriction of $\mathrm{pr}$ to any face of $\Sk(\cX')$ is a homeomorphism onto a face of $\Sk(\cX)$.
 \end{prop}
 \begin{proof} That $\cX'$ is log regular follows from Proposition~\ref{prop:logreg}.
  The underlying set of the fan $F(\cX')$ is the inverse image of the underlying set of  $F(\cX)$ under the projection morphism $\cX'\to \cX$.
  Let $\xi'$ be a point of $F(\cX')$ and denote by $\xi$ its image in $\cX$.
   Let $M$ be the saturation of the coproduct $\mathcal{C}_{F(\cX),\xi}\oplus_{\N}(1/e)\N$, where $e=e(K'/K)$ is the ramification index.
   Then
 the characteristic monoid at $\xi'$ is canonically isomorphic to the quotient of $M$ by its torsion subgroup $M_{\mathrm{tor}}$
   (the proof of~\cite[2.1.1]{nakayama} remains valid for log structures on the Zariski site).
   The natural embeddings of the skeleta $\Sk(\cX)$ and $\Sk(\cX')$ into $X^{\an}$ and $(X')^{\an}$ induce a commutative square
  \[ \xymatrix{
\{\alpha'\in \Hom(M/M_{\mathrm{tor}},\R_{\geq 0})\,|\,\alpha'(\overline{\unif})=e\} \ar[d]_{\mathrm{pr}_{\xi'}} \ar[r] &(X')^{\an} \ar[d]^{\mathrm{pr}}
\\ \{\alpha\in \Hom(\mathcal{C}_{F(\cX),\xi},\R_{\geq 0})\,|\,\alpha(\overline{\unif})=1 \}  \ar[r] & X^{\an} }
\]
   where the map $\mathrm{pr}_{\xi'}$ is defined by restricting $\alpha'$ to $\mathcal{C}_{F(\cX),\xi}$ and dividing the resulting map by $e$.  The map $\mathrm{pr}_{\xi'}$ is a homeomorphism, because every element in the monoid $\R_{\geq 0}$ is uniquely divisible by $e$.
   This shows that $\pr(\Sk(\cX'))=\Sk(\cX)$, and that
    $\mathrm{pr}$ maps the face of $\Sk(\cX')$ corresponding to $\xi'$ homeomorphically onto the face of $\Sk(\cX)$ corresponding to $\xi$.


  It remains to check that $\pr^{-1}(\Sk(\cX))\subset\Sk(\cX')$. This is clear when $K'/K$ is purely inseparable, since $\pr$ is then a homeomorphism.
  Thus we may suppose that $K'$ is separable over $K$. Replacing $K'$ by a finite extension, we can further reduce to the case where $K'$ is a Galois extension of $K$. It follows directly from the definition of the skeleton $\Sk(\cX')$ that it is stable under the action of the Galois group $\mathrm{Gal}(K'/K)$ on $(X')^{\an}$. Since this action is transitive on the fibers of $\pr$, the surjectivity of $\Sk(\cX')\to \Sk(\cX)$ implies that
$\pr^{-1}(\Sk(\cX))\subset\Sk(\cX')$.
\end{proof}
 \begin{prop}\label{prop:compat5}
   In the setting of Proposition~\ref{prop:compat2}, assume that the strata of $\cX_k$ are geometrically connected over $k$, and suppose that $\cX$ is log smooth over $S$ and $\cX_k$ is reduced, or that $K'$ is tamely ramified over $K$ and the multiplicity of each component in $\cX_k$ is a power of $p$. Then $\mathrm{pr}$ induces a homeomorphism $\Sk(\cX')\to \Sk(\cX)$.
 \end{prop}
 \begin{proof}
      By our assumption that the strata of $\cX_k$ are geometrically connected, the conclusion holds when $K'$ is unramified over $K$; thus we may assume that $K'$ is totally ramified over $K$.
    It suffices to show that for every stratum $E$ of
   $\cX_k$,
   the normalization morphism $h\colon \cX'\to \cX\otimes_R R'$ induces an isomorphism $h^{-1}(E)_{\red}\to E$.
   We have already argued in the proof of Proposition~\ref{prop:logreg} that the normalization morphism $h$ is the saturation morphism for the log scheme $\cX\otimes_R R'$ with its divisorial log structure induced by the special fiber.
    By the same argument as in the proof of Lemma 4.1.2 in~\cite{logzeta}, Proposition 2.2.2(3) in~\cite{logzeta} implies that this saturation morphism $h$ is an isomorphism over each stratum, because the multiplicities of the components in $\cX_k$ are prime to the ramification degree of $K'$ over $K$.
    %
  \end{proof}

\subsection{The temperate part of the skeleton}\label{ss:tamepart}
Let $X$ be a smooth and proper $K$-scheme.
\begin{definition}
 We say that a divisorial point $x\in X^{\an}$ is {\em tame} if the field
$\mathscr{H}(x)$ is tamely ramified over $K$; that is, the ramification index of $\mathscr{H}(x)$ over $K$ is prime to $p$, and the residue field of $\mathscr{H}(x)$ is a separable extension of $k$.
 Otherwise, we say that $x$ is {\em wild}.
We define the {\em temperate part} $X^t$ of $X$ to be the closure of the set of tame divisorial points in $X^{\an}$.
\end{definition}

 If $x$ is the divisorial point associated with a normal $R$-model $\cX$ of $X$ and a prime component $E$ of $\cX_k$, then the ramification index of $\mathscr{H}(x)$ over $K$ equals the multiplicity $N$ of $\cX_k$ along $E$, and the residue field of
  $\mathscr{H}(x)$ is $k$-isomorphic to the function field $k(E)$ of $E$. Thus $x$ is tame if and only if $N$ is prime to $p$ and $k(E)$ is separable over $k$.
  If $k$ has characteristic zero, then all divisorial points are tame. Since the set of divisorial points is dense in $X^{\an}$ by \cite[2.4.9]{MuNi}, we then have $X^t=X$.

  We call $X^t$ the temperate, rather than tame, part of $X^{\an}$ because it may contain wild divisorial points; see Example \ref{exam:temperate}.

\begin{definition}
If $\cX$ is a log regular $R$-model of $X$, then we define the temperate part of $\Sk(\cX)$
 as the closure of the set of tame divisorial points in $\Sk(\cX)$, and we denote it by $\Sk^t(\cX)$.
\end{definition}

\begin{lemma}\label{lemm:tameretract}
Let $\cX$ be a log regular proper $R$-model of $X$. Let $x$ be a point of $\widehat{\cX}_\eta$ such that either $x$ is a tame divisorial point, or $\mathscr{H}(x)$ is a tame finite extension of $K$. Then $x'=\rho_{\cX}(x)$ is a tame divisorial point on $\Sk(\cX)$.
\end{lemma}
\begin{proof}
 By the definition of the retraction $\rho_{\cX}$, we know that $\xi'=\spe_{\cX}(x')$ is the generic point of the  unique stratum $U$ of $\cX_k$ containing $\xi=\spe_{\cX}(x)$, and that $|f(x)|=|f(x')|$ for every element $f$ of
 $\mathcal{M}_{\cX,\xi}$. Since $\cX$ is log regular, the stratum $U$ is regular, and the residue field $\kappa(\xi)$ of $U$ at $\xi$ is separable over $k$ because it is contained in the residue field of $\mathscr{H}(x)$. Thus $U$ is smooth over $k$ at $\xi$, and, therefore, also at the generic point $\xi'$, so that $k(\xi')$ is separable over $k$. The residue field of $\mathscr{H}(x')$ is a purely transcendental extension of $k(\xi')$ by Proposition \ref{prop:maxmon}, and thus still separable over $k$.

A point of $\Sk(\cX)$ is divisorial if and only if the value group of its residue field is discrete \cite[2.4.8]{MuNi}. Log regularity of $\cX$ implies that the cospecialization map $\mathcal{C}_{\cX,\xi}\to \mathcal{C}_{\cX,\xi'}$ is an isomorphism~\cite[3.2.1]{logzeta}. This means that we can write every element $g$ of $\mathcal{M}_{\cX,\xi'}$ as the product of an element $f$ in $\mathcal{M}_{\cX,\xi}$ and a unit in $\mathcal{M}_{\cX,\xi'}$; then $|g(x')|=|f(x')|=|f(x)|$. It follows that the value group of $\mathscr{H}(x')$ is contained in that of $\mathscr{H}(x)$, so that $x'$ is divisorial and tame.
\end{proof}

\begin{prop}\label{prop:tamesk}
If $\cX$ is a log regular proper $R$-model of $X$, then $\Sk^t(\cX)=\Sk(\cX)\cap X^t$.
\end{prop}
\begin{proof}
It is obvious that $\Sk^t(\cX)\subset \Sk(\cX)\cap X^t$.
 To prove the converse inclusion, we observe that the open subset $A=\rho_{\cX}^{-1}(\Sk(\cX)\setminus \Sk^t(\cX))$ of $X^{\an}$ contains $\Sk(\cX)\setminus \Sk^t(\cX)$ and does not contain any tame divisorial points by  Lemma \ref{lemm:tameretract}. Thus $A$, and therefore $\Sk(\cX)\setminus \Sk^t(\cX)$, are disjoint from $X^t$.
 \end{proof}

\begin{prop}\label{prop:nonempty}
Let $\cX$ be a log regular proper $R$-model of $X$. Then $\Sk^t(\cX)$ is non-empty if and only if $X$ has a $K^t$-rational point.
\end{prop}
\begin{proof}
 If $X$ has a $K^t$-rational point, then the image of this point under $\rho_{\cX}$ is a tame divisorial point on $\Sk(\cX)$ by Lemma \ref{lemm:tameretract}, so that $\Sk^t(\cX)$ is non-empty.
   Conversely, assume that $\Sk^t(\cX)$ is non-empty. Then
  $X^{\an}$ contains a tame divisorial point $x$.
   We need to prove that $X$ has a $K^t$-rational point. Let $N$ be the ramification index of $\mathscr{H}(x)$ over $K$. By our assumption that $x$ is tame, we know that $N$ is prime to $p$. We can apply a base change to a tamely ramified extension of $K$ of ramification index $N$ to reduce to the case where $N=1$.
      Then the divisorial point $x$ is associated with
   a regular $R$-model $\cY$ of $X$ and a prime component $E$ of $\cY_k$ of multiplicity one.
   Since $x$ is tame, the function field $k(E)$ is separable over $k$, so $E$ has a smooth $k^s$-rational point $\xi$; then $\cY$ is smooth over $R$ at $\xi$, and Hensel's lemma implies that $\xi$ lifts to a closed point $y$ of $X$ whose residue field is unramified over $K$.
   \end{proof}

\begin{prop}\label{prop:tame}
 Assume that $p\geq 2$.
Let $\cX$ be a log regular $R$-model of $X$, and let $x$ be a divisorial point on $\Sk(\cX)$. Then
$x$ is tame if and only if the residue field of $\cX$ at $\spe_{\cX}(x)$ is a separable extension of $k$ and $x$ is $\Z_{(p)}$-integral with respect to the canonical integral $K$-affine structure on the faces of $\Sk(\cX)$.
\end{prop}
\begin{proof}
 We set $\xi=\spe_{\cX}(x)$. Then $\xi$ is a point of $F(\cX)$, and the corresponding face of $\Sk(\cX)$ is the unique face $\sigma$ that contains $x$ in its relative interior. By the definition of the integral $K$-affine structure on $\sigma$, the point $x$ is $\Z_{(p)}$-integral if and only if $\log |f(x)|/\log |\unif|$ lies in $\Z_{(p)}$ for every element $f$ of $\mathcal{M}_{\cX,\xi}$. By Proposition \ref{prop:maxmon}, the value group $|\mathscr{H}(x)^{\times}|$ is generated by the submonoid of $(\R_{>0},\cdot)$ consisting of the elements of the form $|f(x)|$ with $f$ in $\mathcal{M}_{\cX,\xi}$.
 Thus the ramification index of $\mathscr{H}(x)$ over $K$, which is by definition the index of the subgroup $|\unif|^{\Z}$ of $|\mathscr{H}(x)^{\times}|$, is prime to $p$ if and only if $x$ is $\Z_{(p)}$-integral.
 The residue field of $\mathscr{H}(x)$ is a purely transcendental extension of $k(\xi)$ by Proposition \ref{prop:maxmon}, and hence separable over $k$ if and only if $k(\xi)$ is separable over $k$.
  \end{proof}

 We will use Proposition~\ref{prop:tame} to give a more explicit description of the temperate part of the skeleton of a log regular model.
 Let $\cX$ be a log regular $R$-model of $X$, and let $\xi$ be a point of the fan $F(\cX)$. Following~\cite{logzeta}, we define the {\em root index} of $\cX$ at $\xi$ to be the smallest positive integer $\rho$ such that the residue class $\overline{\unif}$ of $\unif$
 in the characteristic monoid $\mathcal{C}_{\cX,x}$ is divisible by $\rho$.

\begin{example}
Let $\cX$ be a regular $R$-model for $X$ such that $\cX_k$ has strict normal crossings, and let $\xi$ be a point of $F(\cX)$.
 As we have seen in Example~\ref{exam:snc}, the characteristic monoid $\mathcal{C}_{F(\cX),\xi}$ is isomorphic to $\N^J$, where $J$ is the set of irreducible components of $\cX_k$ passing through $\xi$. If we denote by $N_j$ the multiplicity of the $j$-th component, for every $j$ in $J$, then  $\overline{\unif}=(N_j)_{j\in J}$.
 Thus the root index of $\cX$ at $\xi$ is the greatest common divisor of the multiplicities $N_j,\,j\in J$.

 This description of the root index does not hold for all log regular proper $R$-models $\cX$. For instance, assume that $\cX$ contains an open subscheme isomorphic to
 $$\cU=\Spec R[u,v]/(uv-\varpi^2)$$ and let $\xi$ be the origin of $\cU_k$.
 Then the characteristic monoid
 $\mathcal{C}_{\cX,\xi}$ is isomorphic to the submonoid of $\N^2$ generated by $(1,0)$ and $(1,2)$, and the residue class of $\unif$ equals $(1,1)$. Thus the root index of $\cX$ at $\xi$ equals $1$, even though the two components of $\cU_k$ that pass through $\xi$ both have multiplicity $2$ (we find these multiplicities by evaluating the primitive generators of the one-dimensional faces of the dual monoid  $\mathcal{C}^{\vee}_{\cX,\xi}$ at the element $\overline{\unif}=(1,1)$).
   \end{example}

\begin{prop}\label{prop:tamecrit}
 Let $\cX$ be a log regular $R$-model of $X$. Let $\xi$ be a point of $F(\cX)$, and let $\sigma$ be the corresponding face of
 $\Sk(\cX)$. We denote by $\rho$ and $k(\xi)$ the root index and the residue field of $\cX$ at $\xi$, respectively.

  If $\rho$ is prime to $p$ and $k(\xi)$ is separable over $k$, then the tame divisorial points are dense in $\sigma$. Otherwise, $\sigma$ does not contain any tame divisorial point.
\end{prop}
\begin{proof}
We may assume that $p\geq 2$, since otherwise the statement is trivial.
 By Proposition~\ref{prop:tame}, it suffices to prove the following claim: if $p$ divides $\rho$, then $\sigma$ does not contain any $\Z_{(p)}$-integral points; and if $p$ does not divide $\rho$, then the $\Z_{(p)}$-integral points are dense in $\sigma$.

   We denote by $d$ the dimension of $\sigma$, and by
  $\mathcal{C}^{\mathrm{gp}}_{\cX,\xi}$ the groupification of the characteristic monoid of $\mathcal{C}_{\cX,\xi}$.
       We can find an isomorphism $\mathcal{C}^{\mathrm{gp}}_{\cX,\xi}\to \Z^{d+1}$ that maps
   $\overline{\unif}$ to $(\rho,0,\ldots,0)$. The induced $\Z$-integral embedding
   $$\Hom(\mathcal{C}_{\cX,\xi}, \R_{\geq 0}) \to \R^{d+1}$$ identifies $\sigma$
 with a $d$-dimensional polytope $P$ in the hyperplane $H$ defined by $z_1=1/\rho$.
  If $\rho$ is divisible by $p$, then $P(\Z_{(p)})$ is empty; otherwise, $P(\Z_{(p)})$ is dense in $P$
  because $P$ has non-empty interior in $H$ and $\Z_{(p)}$ is dense in $\R$.
\end{proof}

\begin{example}\label{exam:temperate}
Assume that $k$ is algebraically closed and of characteristic $p=2$. Let $E$ be an elliptic curve over $K$ of Kodaira-N\'eron reduction type $I^{\ast}_r$, with $r\geq 0$, and let $\mathscr{E}$ be its minimal snc-model (see Figure \ref{fig:temperate}).  By Proposition \ref{prop:tamecrit}, the temperate part $\Sk^t(\mathscr{E})$ of $\Sk(\mathscr{E})$ is the union of the four closed edges in $\Sk(\mathscr{E})$ corresponding to the four singular points in $\mathscr{E}_{k}$ that are contained in a component of multiplicity one.
 Thus $\Sk^t(\mathscr{E})=\Sk(\mathscr{E})$ when $r=0$, and for all $r\geq 0$,
 $\Sk^t(\mathscr{E})$ contains wild vertices of $\Sk(\cE)$. These examples also illustrate that, in the set-up of Proposition \ref{prop:tamecrit},
   $\Sk^t(\cX)$ may contain faces of  $\Sk(\cX)$ corresponding to strata of $\cX_k$ where the root index is not prime to $p$ (in our example, the wild vertices in $\Sk^t(\mathscr{E})$).
\end{example}

\begin{figure}[ht]
  \includegraphics[width=8cm]{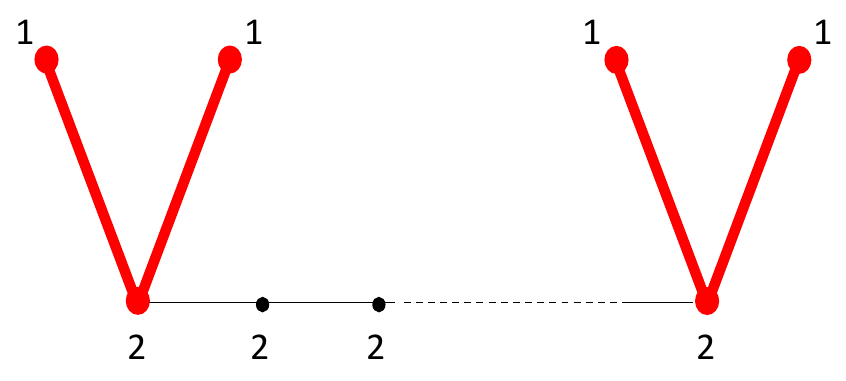}
  \caption{The skeleton of the minimal snc-model $\mathscr{E}$ of an elliptic curve of type $I^{\ast}_r,\,r\geq 0$.
  The vertices are labelled with the multiplicities of the corresponding components in $\mathscr{E}_k$. The number of vertices of multiplicity two is $r+1$ (if $r=0$, then the unique vertex of mutiplicity two has four adjacent edges). The bold part of the graph is the temperate part $\Sk^t(\mathscr{E})$ of the skeleton $\Sk(\mathscr{E})$ if $p=2$.  }
  \label{fig:temperate}
\end{figure}

\begin{prop}\label{prop:logsmoothsk}  Let $n$ be the dimension of $X$, and let $\cX$ be a log smooth $R$-model of $X$.
 Then every $n$-dimensional face of $\Sk(\cX)$ is contained in $\Sk^t(\cX)$.
\end{prop}
\begin{proof}
We endow $\cX$ and $S=\Spec R$ with their standard log structures. We may assume that
 $X$ is of pure dimension $n$, since the connected components of lower dimension never contribute
 $n$-dimensional faces to the skeleton.
 Then the $n$-dimensional faces of $\Sk(\cX)$ correspond bijectively to the zero-dimensional strata of $\cX_k$.
 Let $\xi$ be a zero-dimensional stratum of $\cX_k$.
 The assertion is local for the \'etale topology on $\cX$. Thus by Kato's criterion for logarithmic smoothness (see~\cite[3.5]{kato-log} and~\cite[4.1]{kato-def}),
 we may assume that the morphism of log schemes $\cX\to S$ has a chart of the form
$$\begin{CD}
\cX@>>> \Spec \Z[P]
\\ @VVV @VVV
\\ S @>>> \Spec \Z[\N]
\end{CD}$$
 satisfying the following properties:
\begin{enumerate}
\item the chart $S\to \Spec \Z[\N]$ is induced by the morphism of monoids $\N\to (S,\cdot)$ that maps $1$ to $\unif$;
\item the right vertical morphism is induced by an injective morphism of fine and saturated monoids $\N\to P$ such that the order of the torsion part of $\coker(\Z\to P^{\mathrm{gp}})$ is prime to $p$;
\item the induced morphism $\cX\to S\otimes_{\Z[\N]}\Z[P]$ is smooth.
\end{enumerate}
 Since $\xi$ is a zero-dimensional stratum and $\cX$ is log regular and of dimension $n+1$, the characteristic monoid
$\mathcal{C}_{\cX,\xi}$ has dimension $n+1$.
  The fact that $\cX\to S\otimes_{\Z[\N]}\Z[P]$ is smooth implies that the dimension of $P$ is at most $n+1$.
  Since $\cX\to \Spec \Z[P]$ is a chart, the induced morphism
 $P\to \mathcal{C}_{\cX,x}$ is surjective.
    Therefore, $P/P_{\mathrm{tors}}\to \mathcal{C}_{\cX,x}$ is an isomorphism, where $P_{\mathrm{tors}}$ denotes the monoid of torsion elements in $P$.

     This implies that $P$ has dimension $n+1$, so that $\cX\to S\otimes_{\Z[\N]}\Z[P]$ is \'etale, and the residue field $k(\xi)$ of $\cX$ at $\xi$ is separable over $k$.
  It also follows that the torsion part of the cokernel of the morphism
  $\Z\to \mathcal{C}^{\mathrm{gp}}_{\cX,x}$ that sends $1$ to $\overline{\unif}$ has order prime to $p$. This order is precisely the root index of $\cX$ at $\xi$, because $\mathcal{C}_{\cX,x}$ is saturated. Thus the result follows from Proposition~\ref{prop:tamecrit}.
\end{proof}

\begin{example}
Assume that $k$ is perfect. Let $C$ be a geometrically connected, smooth and proper curve over $K$, and assume that $C$ has  a log smooth proper $R$-model $\cC$.  Proposition~\ref{prop:logsmoothsk} implies that
  $\Sk^t(\cC)$ contains all the edges of $\Sk(\cC)$.
    Thus either $\Sk^t(\cC)$ is empty and $\Sk(\cC)$ is a point, or $\Sk^t(\cC)$ is equal to $\Sk(\cC)$.
    In the former case, $\cC$ is an $snc$-model of $C$, the special fiber $\cC_k$ is irreducible, and the characteristic of $k$ divides the multiplicity of $\cC_k$.

 If $\Sk^t(\cC)$ is empty, then $C(K^t)$ is empty, by Proposition~\ref{prop:nonempty}.
 Theorem 1.1 in~\cite{mitsui-smeets} implies that, in this case, the $\ell$-adic Euler characteristic of $C$ vanishes; in other words, $C$ has genus one.
   It is shown in in~\cite{mitsui-smeets} that there exists a geometrically connected, smooth and proper curve $C$ over a complete discretely valued field $K$ of residue characteristic $p>0$ with algebraically closed residue field such that $C(K^t)$ is
 empty and $C$ has a log smooth proper $R$-model $\cC$.

  Thus there exist curves $C$ with a proper log smooth $R$-model $\cC$ such that $\Sk^t(\cC)$ is empty, but $\Sk^t(\cC)$ is non-empty whenever $C$ is a curve of genus $g\neq 1$ and $\cC$ is a proper log smooth $R$-model of $C$.
 Theorem 1.1 in~\cite{mitsui-smeets} also applies in higher dimensions:
 if $X$ is a geometrically connected smooth proper $K$-scheme with non-vanishing $\ell$-adic Euler characteristic, and $\cX$ is a log smooth proper $R$-model of $X$, then $\Sk^t(\cX)$ is non-empty.

  We do not know any example where $X$ is a geometrically connected smooth proper $K$-scheme,
  $X(K^t)$ is non-empty, and $X$ has a log smooth model $\cX$ for which $\Sk^t(\cX)$ is different from $\Sk(\cX)$.
 \end{example}

 \subsection{Weight functions and log regular models}
Let $X$ be a smooth and proper $K$-scheme of pure dimension, and let $\form$ be an $m$-canonical form on $X$, for some $m>0$, such that $\form$ is not identically zero on any connected component of $X$. We will explain how the calculation of the weight function $\wt_{\form}$ on skeleta can be extended from $snc$-models to log regular models, and deduce some useful properties of the Kontsevich--Soibelman skeleton $\Sk(X,\form)$. Whenever $\cX$ is a log regular proper $R$-model of $X$, we view $\form$ as a rational section of the line bundle $(\omega^{\mathrm{log}}_{\cX/R})^{\otimes m}$ defined in Section \ref{ss:logdiff}, and we denote by $\mathrm{div}_{\cX}(\form)$ the associated Cartier divisor on $\cX$.

\begin{prop}\label{prop:logweight} Assume that $k$ is perfect.
Let $\cX$ be a log regular proper $R$-model of $X$, and let $x$ be a divisorial point contained in the skeleton $\Sk(\cX)$.
 Let $f$ be a local equation at $\spe_{\cX}(x)$ for $\mathrm{div}_{\cX}(\form)$. Then
 $$\wt_{\form}(x)=\frac{\log |f(x)|}{m\cdot \log |\unif|}$$ where $\unif$ is any uniformizer in $R$.
\end{prop}
\begin{proof} We have $\omega^{\mathrm{log}}_{\cX/R}=\omega_{\cX/R}(\cX_{k,\red}-\cX_k)$ by Proposition~\ref{prop:can}\eqref{it:logcan}.
 Thus if $\cX$ is an $snc$-model, the result follows from~\cite[4.3.7]{MuNi}, modulo the renormalization explained in Remark~\ref{rema:weight}.
  We can reduce to that case by means of a toroidal resolution associated with a regular proper subdivision of the fan of $\cX$, using Proposition~\ref{prop:can}\eqref{it:loget}.
\end{proof}

\begin{remark}
In~\cite{MuNi}, the definition of the weight function was extended to all monomial points in $X^{\an}$ (and then further to $X^{\an}$ by means of an approximation procedure, assuming resolution of singularities). The formula in Proposition~\ref{prop:logweight} is valid for all the points in $\Sk(\cX)$, by the same proof.
\end{remark}

Let $\cX$ be a log regular proper $R$-model of $X$. We write $$\cX_k=\sum_{i\in I}N_iE_i$$ where the $E_i$ are the prime components of $\cX_k$, and the $N_i$ are their multiplicities. For every $i\in I$, we denote by $w_i$ the multiplicity of $E_i$ in $\mathrm{div}_{\cX}(\form)$.
 Let $\xi$ be a point of the fan $F(\cX)$ contained in $\cX_k$; thus $\xi$ is the generic point of a stratum of $\cX_k$.
 We say that $\xi$ is {\em $\form$-essential} if the following two properties are satisfied.
\begin{enumerate}
\item For every $j\in I$ such that $\xi$ is contained in $E_j$, we have $$\frac{w_j}{N_j}=\min\{\frac{w_i}{N_i}\,|\,i\in I\}.$$
\item The point $\xi$ is not contained in the Zariski closure of the zero locus of $\form$ on $X$.
\end{enumerate}
 This definition generalizes the one for $snc$-models in~\cite[4.7.4]{MuNi}. We say that a face of $\Sk(\cX)$ is $\form$-essential if the corresponding point of $F(\cX)$ is $\form$-essential.

\begin{prop}\label{prop:sklogmodel}
Assume that $k$ is perfect.  Let $\cX$ be a log regular proper $R$-model for $X$.
The Kontsevich--Soibelman skeleton $\Sk(X,\form)$ is the union of the $\form$-essential faces of $\Sk(\cX)$.
 Moreover, using the notation introduced above, the minimal value of $\wt_{\form}$ on the set of divisorial points in $X^{\an}$ is equal to $$\min\{\frac{w_i}{mN_i}\,|\,i\in I\}.$$
\end{prop}
\begin{proof}
   Let $\cY\to \cX$ be a toroidal resolution of $\cX$, associated with a regular proper subdivision of the fan $F(\cX)$.
 Then $\Sk(X,\form)$ is the union of the $\form$-essential faces of $\Sk(\cY)$, by~\cite[4.7.5]{MuNi}.
 Thus, it suffices to prove that a face of $\Sk(\cY)$ is $\form$-essential if and only if it is contained in a $\form$-essential face of $\Sk(\cX)$.
 This follows directly from Propositions~\ref{prop:can}\eqref{it:loget} and~\ref{prop:logweight}. In particular, $\wt_{\form}$ reaches its minimal value at a vertex of $\Sk(\cX)$; the value of $\wt_{\form}$ at the vertex corresponding to $E_i$ is precisely $w_i/mN_i$.
\end{proof}

\subsection{Behavior under ground field extension}
Consider a finite extension $K'/K$, of ramification index $e=e(K'/K)$. Denote by $R'$ the integral closure of $R$ in $K'$. Set $X'=X\otimes_KK'$, let
\begin{equation*}
  \mathrm{pr}\colon (X')^{\an}\to X^{\an}
\end{equation*}
be the canonical projection map, and let $\form'$ be the pullback of $\form$ to $X$.
\begin{prop}\label{prop:skbasech}
  We assume that $k$ is perfect.
\begin{enumerate}
\item \label{it:skbasech-tame} Assume that $K'$ is tame over $K$. Let $x$ be a divisorial point of $X^{\an}$ and $x'$ a point in $\mathrm{pr}^{-1}(x)$. Then $x'$ is divisorial, and
$\wt_{\form'}(x')=e\cdot \wt_{\form}(x)$.
\item \label{it:skbasech-logsm} Assume that $X$ has a log smooth proper $R$-model $\cX$. Then the normalization $\cX'$ of $\cX\otimes_R R'$ is a log smooth proper $R'$-model, and $$\wt_{\form'}=e\cdot \wt_{\form}\circ \mathrm{pr}$$ on $\Sk(\cX')$.
\end{enumerate}
In both cases, we have $\Sk(X',\form')=\pr^{-1}(\Sk(X,\form))$.
\end{prop}
\begin{proof}
~\eqref{it:skbasech-tame} Let $\cX$ be a normal $R$-model of $X$ and let $E$ be a prime component of $\cX_k$ such that $x$ is the divisorial point associated with the pair $(\cX,E)$. Shrinking $\cX$ around the generic point of $E$, we may assume that $\cX$ and $E$ are regular, and that $\cX_k$ is irreducible. Then $\cX$ is log regular (with respect to its standard log structure).
 Denote by $\cX'$ the normalization of $\cX\otimes_R R'$. Then $x'$ is the divisorial point associated with $\cX'$ and some prime component $E'$ of $\cX'$ dominating $E$.
 Proposition~\ref{prop:can} implies that, locally around the generic point of $E'$, the line bundle $\omega_{\cX'/R'}(\cX'_k-\cX'_{k,\red})$ is the pullback of $\omega_{\cX/R}(\cX_k-\cX_{k,\red})$ to $\cX'$. It follows that $\wt_{\form'}(x')=e\cdot \wt_{\form}(x)$. In particular, $\wt_{\form'}$ is minimal at $x'$ if and only if
  $\wt_{\form}$ is minimal at $x$.
   By definition, $\Sk(X,\form)$ is the closure in $X^{\an}$ of the set of divisorial points where $\wt_{\form}$ is minimal, and the analogous statement holds for
   $\Sk(X',\form')$. Since $\mathrm{pr}$ is open (see~\cite[3.2.7]{BerkIHES}),
it follows that $\Sk(X',\form')$ is the inverse image of $\Sk(X,\form)$ in $(X')^{\an}$.

\eqref{it:skbasech-logsm}
Proposition~\ref{prop:logreg} shows that $\cX'$ is a proper log smooth $R'$-model of $X'$. We know by Proposition~\ref{prop:sklogmodel} that $\Sk(X,\form)$ is a union of faces of $\Sk(\cX)$, and
 $\Sk(X',\form')$  is a union of faces of $\Sk(\cX')$.
 Moreover, $\Sk(\cX')$ is the inverse image of $\Sk(\cX)$ under the projection morphism $\mathrm{pr}$, by Proposition~\ref{prop:compat2}.
  Propositions~\ref{prop:can}\eqref{it:canbc} and~\ref{prop:logweight} imply that
  $$\wt_{\form'}=e\cdot \wt_{\form}\circ \mathrm{pr}$$ on $\Sk(\cX')$.
   Thus the set of divisorial points of minimal weight in $\Sk(\cX')$ is the preimage of the set of divisorial points of minimal weight in $\Sk(\cX)$.
  It follows that $\Sk(X',\form')$ is the preimage of $\Sk(X,\form)$ in $(X')^{\an}$.
\end{proof}

 The last assertion of Proposition~\ref{prop:skbasech} is false without the assumption that $K'$ is tame over $K$ or $X$ has a log smooth model; let us illustrate this by means of an explicit example.

\begin{example}\label{exam:wild}
 Let $k$ be an algebraically closed field of characteristic $p=2$, let $R$ be the ring $W(k)$ of Witt vectors, and let $K$ be the quotient field of $R$.
 Let $E$ be an elliptic curve over $K$ of Kodaira-N\'eron reduction type $I_{r}^{\ast}$, with $r>0$, and let $\form$ be a volume form on $E$.
Let $\cE$ be the minimal $snc$-model of $E$. Using the triviality of $\omega_{\cE/R}$, one immediately checks that $\Sk(E,\form)$ is homeomorphic to $[0,1]$ (see Figure \ref{fig:wild}; this also follows from~\cite[3.3.13]{BaNi}).
 However, there are examples of such curves $E$ that acquire good reduction over a quadratic extension $K'$ of $K$; they have been classified in~\cite[2.8]{lor}. In those cases,   $\Sk(E',\form')$ is a point.
\end{example}

\begin{figure}[ht]
  \includegraphics[width=8cm]{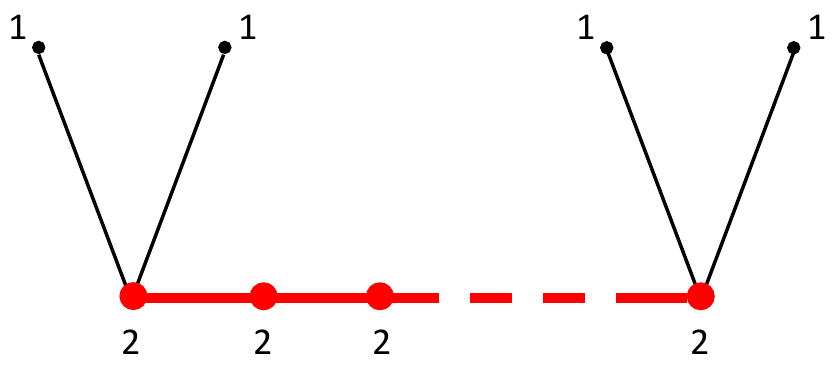}
  \caption{The skeleton of the minimal snc-model $\mathscr{E}$ of an elliptic curve $E$ of type $I^{\ast}_r$, $r>0$.
  The vertices are labelled with the multiplicities of the corresponding components in $\mathscr{E}_k$. The number of vertices of multiplicity two is $r+1$. The bold part of the graph is $\Sk(E,\form)$, for any volume form $\form$ on $E$.  }
  \label{fig:wild}
\end{figure}

\section{Lebesgue measures on skeleta}\label{sec:lebesgue}
In this section we study Lebesgue type measures on the skeleta of log regular models, induced by the canonical integral affine structures on the faces.
Throughout, $X$ is a smooth and proper $K$-scheme of pure dimension.
\subsection{Definitions}\label{ss:lebesgue}
Let $\cX$ be a log regular proper $R$-model for $X$,
$\xi$ the generic point of a stratum of $\cX_k$, and $\tau$ the corresponding face of $\Sk(\cX)$. Write $d$ for the dimension of $\tau$.
The face $\tau$ carries an integral $K$-affine structure induced by the characteristic monoid $\mathcal{C}_{F(\cX),\xi}$ (see Section \ref{ss:logskeleta}). This integral affine structure gives rise to a  measure on $\tau$: the real affine space $$\{\alpha\in \Hom(\mathcal{C}_{F(\cX),\xi},\R)\,|\,\alpha(\overline{\unif})=1\}$$ spanned by $\tau$ is a torsor under the action of
the real vector space  $$\{\alpha\in \Hom(\mathcal{C}_{F(\cX),\xi},\R)\,|\,\alpha(\overline{\unif})=0\}$$ by translation, and the latter space carries a canonical translation-invariant  measure
 $|dv_1\wedge \ldots \wedge dv_d|$ where $v_1,\ldots,v_d$ is any basis for the lattice
$$\{\alpha\in \Hom(\mathcal{C}_{F(\cX),\xi},\Z)\,|\,\alpha(\overline{\unif})=0\}.$$
 This measure induces a measure on $\tau$, which we call the integral Lebesgue measure (when $\tau$ has dimension $d=0$, we interpret this definition as the Dirac measure on $\tau$). Equivalently, the integral Lebesgue measure on $\tau$ is the limit of the discrete measures

 $$\frac{1}{e^d}\sum_{x\in \tau((1/e)\Z)}\delta_x$$ over all positive integers $e$, ordered by divisibility, where the sum is taken over the $(1/e)\Z$-integral points in $\tau$ and $\delta_x$ denotes the Dirac measure at $x$.

 For every nonnegative integer $d$, we denote by $\Leb^{(d)}_{\cX}$ the measure supported on the union of the $d$-dimensional faces in $\Sk(\cX)$ whose restriction
 to each $d$-dimensional face is the integral Lebesgue measure.

  \begin{prop}\label{prop:compat3}
  Let $\cX$ and $\cY$ be log regular proper $R$-models of $X$.
  Then, for every nonnegative integer $d$, the measures $\Leb^{(d)}_{\cX}$ and $\Leb^{(d)}_{\cY}$ coincide on the intersection of the $d$-skeleta of $\Sk(\cX)$ and $\Sk(\cY)$.
\end{prop}
\begin{proof}

  It follows directly from the definitions that the retraction $$\rho_{\cX}\colon X^{\an}\to \Sk(\cX)$$ is piecewise $K$-affine
 on every face of $\Sk(\cY)\subset X^{\an}$. The analogous property holds when we swap $\cX$ and $\cY$. Thus the piecewise $K$-affine structures on $\Sk(\cX)$ and $\Sk(\cY)$ agree on the intersection $\Sk(\cX)\cap \Sk(\cY)$, so that the measures $\Leb_{\cX}^{(d)}$ and $\Leb_{\cY}^{(d)}$ coincide on $\Sk^{(d)}(\cX)\cap \Sk^{(d)}(\cY)$. Alternatively, we can reduce to the case where $\cX$ and $\cY$ are $snc$-models by means of regular proper subdivisions of their fans~\cite[10.4]{kato}, and apply~\cite[3.2.4]{MuNi}.
\end{proof}
\begin{example}\label{exam:torres}
Let $\cY\to \cX$ be a toroidal modification, corresponding to a proper subdivision of the fan of $\cX$~\cite[\S9]{kato}.
 Then the polyhedral complex $\Sk(\cY)$ is a subdivision of $\Sk(\cX)$, and both skeleta are equal as subsets of $X^{\an}$.
 If $d$ is the dimension of $\Sk(\cX)$, then $\Leb^{(d)}_{\cX}=\Leb^{(d)}_{\cY}$. We can always choose our subdivision of the fan of $\cX$ in such a way that $\cY$ is an $snc$-model of $X$~\cite[10.4]{kato}. In that case, we call $\cY\to \cX$ a {\em toroidal resolution} of $\cX$.
\end{example}
\subsection{Behavior under ground field extension}
Now consider a finite extension $K'$ of $K$, and set $X':=X\otimes_KK'$.
We have a canonical surjective map $\pr\colon (X')^{\an}\to X^{\an}$.
Let $\cX$ be a log regular proper $R$-model of $X$, and let $\cX'$ be the normalization of $\cX\otimes_R R'$, where $R'$ is the integral closure of $R$ in $K'$.
Assume that $\cX$ is log smooth over $S=\Spec\,R$, or that $K'$ is tamely ramified over $K$.
By Proposition~\ref{prop:logreg}, $\cX'$ is a log regular proper $R'$-model of $X'$. Further, Proposition~\ref{prop:compat2} shows that $\pr^{-1}(\Sk(\cX))=\Sk(\cX')$
and that $\pr$ maps any face $\tau'$ of $\Sk(\cX')$ homeomorphically onto a face $\tau$ of $\Sk(\cX)$.
\begin{prop}\label{prop:compat4}
  In the situation above, $\pr$ identifies the measures $\Leb^{(d)}_{\cX'}|_{\tau'}$
  and $e^d\cdot \Leb^{(d)}_{\cX}|_{\tau}$ for every $d\geq 0$,
  where $e=e(K'/K)$ is the ramification index.
\end{prop}
\begin{proof} We use the notation of the proof of Proposition~\ref{prop:compat2}. The restriction of $\pr$ to $\tau'$ corresponds to the map $\pr_{\xi'}$ on characteristic monoids. The assertion now follows from the fact that $\pr_{\xi'}$ is given by restriction and then multiplication by $e^{-1}$.
\end{proof}
If $\pr$ maps $\Sk(\cX')$ homeomorphically onto $\Sk(\cX)$, we therefore have
that $\pr_*\Leb^{(d)}_{\cX'}=e^d\Leb^{(d)}_\cX$ for every $d$. This is the case
in the setting of Proposition~\ref{prop:compat5}, for example.
In general, we only have the inequality $\pr_*\Leb^{(d)}_{\cX'}\ge e^d\Leb^{(d)}_\cX$. Below we define a weighted Lebesgue measure that behaves better under ground field extension.

\subsection{Tame degree function}\label{ss:tamedeg}
In the setting of Proposition~\ref{prop:compat5}, if $\cX_k$ is not reduced or $k$ is not separably closed, then open faces of the skeleton $\Sk(\cX)$ may split into multiple copies by taking the inverse image
in $(X')^{\an}$. In order to encode the number of copies, we introduce the following invariant.
\begin{definition}
The \emph{tame degree function} $\tdeg_K\colon X^{\an}\to \N\cup \{\infty\}$ maps each point $x$ of $X^{\an}$ to the extension degree of the tame closure of $K$ in $\mathscr{H}(x)$.
\end{definition}

\begin{lemma}\label{lemma:tame}
Let $\cX$ be log smooth proper $R$-model for $X$. Then for every point $x$ of $\Sk(\cX)$, the tame closure of $K$ in $\mathscr{H}(x)$ is separably closed in $\mathscr{H}(x)$.
\end{lemma}
\begin{proof}
 The result is trivial when $k$ has characteristic zero, so we may assume that $k$ has positive characteristic $p$. By base change to the completion of the maximal unramified extension of $K$, we can reduce to the case where $k$ is separably closed.
Let $x$ be a point of $\Sk(\cX)$ and assume that $K$ has a wild finite separable extension $L$ inside $\mathscr{H}(x)$. We will deduce a contradiction.

Replacing $K$ by a finite tame extension in $L$, we may assume that $[L\colon K]$ is a power of $p$. Let $e$ be the least common multiple of the multiplicities of the components in the special fiber $\cX_k$. We claim that there exists a finite extension $K'$ of $K$ of ramification degree divisible by $e$ that is linearly disjoint from $L$.
 Let $R'$ be the integral closure of $R$ in $K'$, and let $\cX'$ be the normalization of $\cX\otimes_R R'$. By Proposition~\ref{prop:saturated}, $\cX'$ is a log smooth proper $R'$-model of $X'=X\otimes_K K'$ with reduced special fiber.
 Proposition~\ref{prop:compat2} implies that for each point $x'$ of $(X')^{\an}$ that lies above $x$, the field $K'$ is separably closed in $\mathscr{H}(x')$, since the point $x'$ does not split under further extensions of $K'$. This contradicts the fact that the $K'$-field $K'L$ embeds into $\mathscr{H}(x')$.

 It remains to prove our claim. Since we can freely replace $K$ by a finite tame extension, we may assume that $e$ is a power of $p$ by Proposition \ref{prop:saturated}, and that $K$ contains all the $p$-th roots of unity in $K^a$. We may also assume that $K$ has characteristic zero, since otherwise, we can simply take $K'=K[x]/(x^e-\unif)$ for any uniformizer $\unif$ in $R$. This is a purely inseparable extension of $K$, and thus certainly linearly disjoint from
the separable extension $L$.

  So assume that $K$ has characteristic zero and contains $\mu_p(K^a)$, and that $e$ is a power of $p$. Since $L$ contains only finitely many extensions of $K$, it is enough to show that $K$ has infinitely many ramified extensions of degree $p$, up to $K$-isomorphism; the result then follows from induction on $e$. For every $\alpha$ in $R^{\times}$, we denote by $K_\alpha$ the splitting field of the Eisenstein polynomial $T^p-\unif-\alpha\unif^2$ in $K^a$. This is a totally ramified Galois extension of degree $p$.
 We will prove that, for all elements $\alpha$ and $\beta$ in $R^{\times}$ whose residues $\overline{\alpha}$ and $\overline{\beta}$ are distinct in $k^{\times}$, the fields $K_\alpha$ and $K_\beta$ are linearly disjoint over $K$ (note that $k$ is infinite because we have reduced to the case where it is separably closed). Assume that $K_\alpha=K_\beta$.
 Then, by Kummer theory, the element $$\frac{1+\alpha\unif}{1+\beta\unif}=1+(\alpha-\beta)\unif + \ldots $$ is a $p$-th power in $R$. The field $K$ is absolutely ramified because it contains $\mu_p(K^a)$. Thus the ring $R/(\unif^2)$ has characteristic $p$, and we find that $(\alpha-\beta)\unif$ is a $p$-th power in this ring. This is absurd, because $(\alpha-\beta)\unif$ is a uniformizer of $R$.
 \end{proof}

    \begin{prop}\label{prop:tamedeg}
    Let $\cX$ be a log regular proper $R$-model of $X$.
    \begin{enumerate}
    \item \label{it:fib} Let $K'$ be a finite extension of $K$.  Assume either that $\cX$ is log smooth over $S=\Spec\,R$ or that $K'$ is a tame extension of $K$. Let $x$ be a point in $\Sk(\cX)$ and denote by $F_x$ the fiber of $(X\otimes_K K')^{\an}$ above $x$. Then we have
        $$\tdeg_K(x)=\sum_{x'\in F_x}\tdeg_{K'}(x').$$
             \item \label{it:one} Assume that the multiplicity of each component of $\cX_k$ is a power of $p$, and that the strata of $\cX_k$ are geometrically connected over $k$. Then $\tdeg_K=1$ on $\Sk(\cX)$.
                         \item \label{it:split} There exists a tame finite extension $L$ of $K$ such that, for
        every tame finite extension $K'$ of $L$ and every point $x$ of $\Sk(\cX)$, the tame degree $\tdeg_K(x)$ at $x$ equals the number of points in $(X\otimes_K K')^{\an}$ lying above $x$.
    \item \label{it:const} The tame degree function $\tdeg_K$ is finite and constant on each open face of
    $\Sk(\cX)$.
    \end{enumerate}
    \end{prop}
    \begin{proof}
\eqref{it:fib} We may assume that $K'$ is separable over $K$, since purely inseparable extensions of the base field have no effect on the tame degree.
Denote by $L$ the tame closure of $K$ in $\mathscr{H}(x)$. Then $L\otimes_K K'$ is a product $L'_1\times \ldots \times L'_r$ of tame extensions of $K'$.
 Moreover, $\mathscr{H}(x)$ is linearly disjoint from $L'_i$ over $L$, for every $i$: if $K'$ is tame over $K$, this is a direct consequence of the definition of $L$; if
  $\cX$ is log smooth, it follows from the fact that $L$ is separably closed in $\mathscr{H}(x)$, by Lemma~\ref{lemma:tame}.
 Thus $F_x$ consists of precisely $r$ points, with residue fields $L'_i\otimes_L \mathscr{H}(x)$ for $i=1,\ldots,r$, and $L'_i$ is the tame closure of $K'$ in $L'_i\otimes_L \mathscr{H}(x)$. Now the result follows from the fact that the degree of $L$ over $K$ is the sum of the degrees of the fields $L'_i$ over $K'$.

~\eqref{it:one} Let $K'$ be a tame finite extension of $K$, and denote by $R'$ the integral closure of $R$ in $K'$. Let $\cX'$ be the normalization of $\cX\otimes_R R'$.
  By Proposition~\ref{prop:compat3}, the pullback to $\Sk(\cX)$ of the projection morphism $(X\otimes_K K')^{\an}\to X^{\an}$ is a homeomorphism
  $\Sk(\cX')\to \Sk(\cX)$. Thus when $x$ is a point of $\Sk(\cX)$, the fiber $F_x=\mathcal{M}(K'\otimes_K \mathscr{H}(x))$ is a point.
   It follows that $K$ has no non-trivial tame finite extensions in $\mathscr{H}(x)$.

~\eqref{it:split}
 By Corollary~\ref{coro:ppower}, we can find a tame finite extension $L$ of $K$ such that the normalized base change of $\cX$ satisfies the conditions of~\eqref{it:one}. Now the assertion
 follows immediately from~\eqref{it:fib} and~\eqref{it:one}.

~\eqref{it:const} Let $L$ be a field satisfying the properties in point~\eqref{it:split}.
  Denote by $R'$ the integral closure of $R$ in $L$. Let $\cX'$ be the normalization of $\cX\otimes_R R'$. By Proposition~\ref{prop:compat2}, the preimage under $$(X\otimes_K L)^{\an}\to X^{\an}$$ of each open face $\tau$ of $\Sk(\cX)$ is a disjoint union of homeomorphic copies of $\tau$. In particular, all the fibers over points in $\tau$ have the same cardinality. By point~\eqref{it:split}, this cardinality equals the value of the tame degree function on $\tau$.
\end{proof}

\subsection{Stable Lebesgue measures}
We now use the tame degree function to define weighted Lebesgue measures that behave well under base change.
\begin{definition}\label{def:stablemeas}
Let $\cX$ be a log regular proper $R$-model of $X$. For every nonnegative integer $d$, we define the
 stable measure on the $d$-skeleton of $\Sk(\cX)$ to be the measure
$$\Leb_{\cX}^{(d),\st}=\tdeg_K \cdot  \Leb_{\cX}^{(d)},$$
where $\tdeg_K$ is the tame degree function on $X^{\an}$.
\end{definition}
\begin{prop}\label{prop:stmeasure1}
For every log regular proper $R$-model $\cX$ of $X$ and every $d\geq 0$, we have
$$\Leb_{\cX}^{(d)} \leq \Leb_{\cX}^{(d),\st}.$$
 If $\cY$ is another log regular proper $R$-model of $X$, then $\Leb_{\cX}^{(d),\st}$ and $\Leb_{\cY}^{(d),\st}$ coincide on the intersection of the $d$-skeleta of $\Sk(\cX)$ and $\Sk(\cY)$. If the strata of $\cX_k$ are geometrically connected over $k$, and the multiplicity of each component in $\cX_k$ is a power of $p$, then $\Leb_{\cX}^{(d),\st}=\Leb_{\cX}^{(d)}$.
\end{prop}
\begin{proof}
 The inequality $\Leb_{\cX}^{(d)} \leq \Leb_{\cX}^{(d),\st}$ is obvious from the definition of the stable measure.
The compatibility of the stable measure with a change of model immediately follows from Proposition~\ref{prop:compat3} and the intrinsic nature of the tame degree function $\tdeg_K$ on $X^{\an}$.  If the strata of $\cX_k$ are geometrically connected over $k$, and the multiplicity of each component in $\cX_k$ is a power of $p$, then $\tdeg=1$ on $\Sk(\cX)$ by Proposition~\ref{prop:tamedeg}, so that $\Leb_{\cX}^{(d),\st}=\Leb_{\cX}^{(d)}$.
\end{proof}
\begin{prop}\label{prop:stmeasure2}
 Let $K'$ be a finite extension of $K$ of ramification degree $e$, and let $R'$ be the integral closure of $R$ in $K'$. Denote by $\cX'$ the normalization of $\cX\otimes_R R'$.
 Assume either that $\cX$ is log smooth over $S=\Spec\,R$ or that $K'$ is a tame extension of $K$. Then $\Leb^{(d),\st}_{\cX}$ is the pushforward of the measure $(1/e)^d\Leb^{(d),\st}_{\cX'}$ on the $d$-skeleton of $\Sk(\cX')$. If $k$ is separably closed and $e$ is divisible by the prime-to-$p$ parts of the multiplicities of all the components in $\cX_k$, then $\Leb^{(d),\st}_{\cX}$ is the pushforward of  $(1/e)^d\Leb^{(d)}_{\cX'}$.
\end{prop}
\begin{proof}
The compatibility of the stable measure with base change follows from Propositions~\ref{prop:compat4} and~\ref{prop:tamedeg}.
If $e$ is divisible by the prime-to-$p$ parts of the multiplicities of all the components in $\cX_k$, then the multiplicities of the components of $\cX'_k$  are powers of $p$,
 by Proposition~\ref{prop:saturated}. Thus if we also assume that $k$ is separably closed, then $\Leb^{(d),\st}_{\cX'}=\Leb^{(d)}_{\cX'}$, and
  $\Leb^{(d),\st}_{\cX}$ is the pushforward of  $(1/e)^d\Leb^{(d)}_{\cX'}$.
\end{proof}

\subsection{Example: faces of dimension one}\label{sec:curve}
Let $X$ be a smooth and proper $K$-scheme, and let $\cX$ be a log regular proper $R$-model of $X$.
 An equivalent way of describing the integral Lebesgue measure $\Leb^{(1)}_{\cX}$ on $\Sk(\cX)$
 is to consider the corresponding metric on each one-dimensional face $\tau$. Let $\xi$ be the point in the fan $F(\cX)$ corresponding to the face $\tau$; in other words,
 $\xi$ is the generic point of the stratum of $\cX_k$ corresponding to $\tau$. We denote by $\overline{\varpi}$ the image of the uniformizer $\unif\in R$ in the characteristic monoid $\mathcal{C}_{\cX,\xi}$.
Let $\rho$ be the largest positive integer that divides $\overline{\varpi}$; this is the root index of $\cX$ at $\xi$ defined in Section \ref{ss:tamepart}.
  Let $v_1$, $v_2$ be primitive generators of the rays of the two-dimensional real cone $$\Hom(\mathcal{C}_{\cX,\xi},\R_{\geq 0}).$$
  They correspond to the irreducible components $E_1$ and $E_2$ of $\cX_k$ that pass through the point $\xi$, and the multiplicity of $E_i$ in $\cX_k$ is equal to
  $N_i=\langle \overline{\varpi},v_i \rangle$ for $i=1,2$. Finally, we define the determinant of the monoid $\mathcal{C}_{\cX,\xi}$ to be the absolute value of the determinant $\det(v_1,v_2)$.
  A straightforward computation shows that the
 lattice length of the face $\tau$ is equal to
 $$\ell(\tau)= \frac{\rho\cdot \det(\mathcal{C}_{\cX,x})}{N_1\cdot N_2}.$$
 If $\cX$ is an $snc$-model, then $\mathcal{C}_{\cX,x}\cong \N^2$, and the formula simplifies to
 $$\ell(\tau)= \frac{\gcd(N_1,N_2)}{N_1\cdot N_2}=\frac{1}{\lcm(N_1,N_2)}.$$

If $X$ is a curve, then, as a topological space, $\Sk(\cX)$ is nothing but the dual graph of $\cX_k$, and our calculation shows that the metric on $\Sk(\cX)$ corresponding to the integral Lebesgue measure $\Leb^{(1)}_{\cX}$ coincides with the {\em stable metric} from~\cite[A.1]{BaNi}.

\subsection{Measures on Kontsevich--Soibelman skeleta}
Let $X$ be a smooth and proper $K$-scheme of pure dimension.
 If $A$ is a subset of $X^{\an}$ that can be realized as a union of faces in $\Sk(\cX)$, for some log regular proper $R$-model $\cX$ of $X$, then the set $A$ carries two natural measures: if we denote by $d$ the dimension of $A$, then we can consider the integral Lebesgue measure $\Leb_A=\Leb^{(d)}_{\cX}|_A$ and the stable measure $\Leb^{\st}_A=\Leb^{(d),\st}_{\cX}|_A$. These measures are independent of the choice of $\cX$, by Proposition~\ref{prop:compat3}. We apply this construction in the following cases.

 \begin{definition}\label{defi:skmeasures}
   Let $\form$ be a pluricanonical form on $X$ that is not identically zero on any connected component of $X$.
 Assume that $X$ has a log regular proper $R$-model $\cX$.

 \begin{enumerate}
\item The Lebesgue measure $\Leb_{\Sk(X,\form)}$ and the stable measure $\Leb^{\st}_{\Sk(X,\form)}$ are
 the restrictions of the measures $\Leb_{\cX}^{(d)}$ and $\Leb^{(d),\st}_{\cX}$ to $\Sk(X,\form)$, where $d$ is the dimension of $\Sk(X,\form)$.

\item
 Likewise, the Lebesgue measure $\Leb_{\Sk(X)}$ and the stable measure $\Leb^{\st}_{\Sk(X)}$ are
 the restrictions of the measures $\Leb_{\cX}^{(d)}$ and $\Leb^{(d),\st}_{\cX}$ to $\Sk(X)$, where $d$ is the  dimension of $\Sk(X)$.

\item  We define the temperate parts $\Sk^t(X)$ and $\Sk^t(X,\form)$ as the closures of the sets of tame divisorial points in $\Sk(X)$ and $\Sk^t(X,\form)$, respectively. These are unions of faces of $\Sk(\cX)$, by Proposition~\ref{prop:tamecrit}.
     The Lebesgue measure $\Leb_{\Sk^t(X)}$ and the stable measure $\Leb^{\st}_{\Sk^t(X)}$ are
 the restrictions of the measures $\Leb_{\cX}^{(d)}$ and $\Leb^{(d),\st}_{\cX}$ to $\Sk(X)$, where $d$ is the  dimension of $\Sk^t(X)$. The measures $\Leb_{\Sk^t(X)}$ and $\Leb^{\st}_{\Sk^t(X,\form)}$ are defined analogously.
  \end{enumerate}
 \end{definition}
 These definitions are independent of the choice of $\cX$.
  If $k$ has characteristic zero, $X$ is a geometrically connected, smooth and projective $K$-scheme, and the canonical bundle
  $\omega_X$ is trivial, then $\Sk(X)$ is a connected pseudo-manifold with boundary, by~\cite[4.2.4]{NiXu}. In particular, if we set $d=\dim(\Sk(X))$ and if $\cX$ is a log regular proper $R$-model of $X$, then $\Sk(X)$ is a union of $d$-dimensional faces of $\Sk(\cX)$. Thus, in this case, $\Sk(X)$ is equal to the support of the measures $\Leb_{\Sk(X)}$ and $\Leb^{\st}_{\Sk(X)}$. In the general case, the measures in Definition~\ref{defi:skmeasures} only detect the top dimensional parts of $\Sk(X,\form)$ and $\Sk(X)$.

\begin{remark}
We have shown in Proposition \ref{prop:tamesk} that $\Sk^t(\cX)=\Sk(\cX)\cap X^t$ when $\cX$ is a log regular proper $R$-model of $X$. However, $\Sk^t(X,\form)$ can be strictly smaller than $\Sk(X,\form)\cap X^t$, and $\Sk^t(X)$ can be strictly smaller than $\Sk(X)\cap X^t$.
This happens, for instance, when $k$ has characteristic $2$ and $E$ is an  elliptic curve of reduction type $I_r^{\ast}$, $r\geq 0$ (see Examples \ref{exam:temperate} and \ref{exam:wild}):  we have $\Sk(E)=\Sk(E,\form)$ for any volume form $\form$ on $E$, and $\Sk(E)$ does not contain any tame divisorial points, so that $\Sk^t(E)=\Sk^t(E,\form)=\emptyset$. On the other hand, $\Sk(E)\cap E^t=\Sk(E)$ when $n=0$, and $\Sk(E)\cap E^t$ consists of the two endpoints of $\Sk(E)$ when $n>0$.
\end{remark}

\section{Convergence of non-archimedean pluricanonical measures}\label{sec:padic}
In this section we prove a precise version of the Main Theorem in the Introduction.

\subsection{Toy example: the Lang--Weil estimates}\label{sec:langweil}
 As a toy model for our results over non-archimedean local fields, we first explain how the asymptotic Lang--Weil estimates can be viewed as a version of our convergence results over finite fields. Let $F$ be a finite field of cardinality $q$, and let $Z$ be an integral $F$-scheme of finite type, of dimension $n$. We endow $F$ with its trivial absolute value, and denote by $Z^{\mathrm{an}}$ the $F$-analytic space associated with $Z$.
 We fix an algebraic closure $F^a$ of $F$ and we denote by $\mathcal{E}^a_F$ the set of finite field extensions of $F$ in $F^a$, ordered by inclusion. By the Lang--Weil estimates, we know that
 $$|Z(F')-c_Zq^{n[F'\colon F]}|= O(q^{n[F'\colon F]-1/2})$$
 as $F'$ ranges through $\mathcal{E}^a_F$, where $c_Z$ denotes the number of geometric irreducible components of $Z$. Note that $c_Z$ is equal to the degree of the algebraic closure of $F$ in the function field of $Z$; this invariant plays the role of the tame degree of $Z$ at its generic point.

 \begin{prop}\label{prop:langweil}
 Let $X$ be an integral proper $F$-scheme, of dimension $n$.
  Denote by $\eta$ the point of $X^{\mathrm{an}}$ corresponding to the trivial absolute value on the function field of $X$, and by $c_X$ the number of geometric irreducible components of $X$.
 As $F'$ ranges through $\mathcal{E}_F^a$, the pushforward to $X^{\mathrm{an}}$ of the normalized counting measure
 $$q^{-n[F'\colon F]}\sum_{x\in X(F')}\delta_x$$ on $X(F')$
 converges to $c_X\delta_{\eta}$. Here
 $\delta_x$ and $\delta_\eta$ denote the Dirac measures at $x$ and $\eta$, respectively.
\end{prop}
\begin{proof}
In order to prove the convergence of the measures, we first define a suitable class of test functions. Let $h\colon Y\to X$ be a proper birational morphism and let $D$ be an effective Cartier divisor on $Y$ such that the restriction of $h$ to $Y\setminus D$ is an open immersion. With such a pair $(Y,D)$ we associate a continuous function
$$\phi_{Y,D}\colon X^{\an}\to \R$$ that is defined in the following way.
If $x$ is a point of the image of $D^{\an}$ under $h^{\an}$, then $\phi_{Y,D}(x)=0$.
 Otherwise, the preimage of $x$ under $h^{\an}$ consists of a unique point $y$ in $(Y\setminus D)^{\an}$, and we set  $\phi_{Y,D}(x)=|f(y)|$ where $f$ is a local equation for $D$ at the point $\spe_{Y}(y)$ (the center of $y$ on $Y$).

 We call a function of the form $\phi_{Y,D}$ a {\em model function} for $X$.
 We denote by $\mathcal{T}(X)$ the real sub-vector space of $\mathcal{C}^0(X^{\an},\R)$ generated by the model functions.
  The set of model functions is closed under multiplication: if $\phi_{Y,D}$ and $\phi_{Y',D'}$ are two model functions for $X$, then $\phi_{Y,D}\cdot \phi_{Y',D'}$ is the model function associated with the sum of the pullbacks of $D$ and $D'$ to $Y\times_X Y'$.
Moreover, the constant function $1$ on $X^{\an}$ is the model function associated with the zero divisor on $X$.
  Thus $\mathcal{T}(X)$ is a sub-$\R$-algebra of $\mathcal{C}^0(X^{\an},\R)$.

 We will now show that the model functions for $X$ separate points on $X^{\an}$. Let $x$ and $x'$ be distinct points of $X^{\an}$, and denote by $\spe_X\colon X^{\an}\to X$ the specialization map that sends a point of $X^{\an}$ to its center on $X$.
 If $\spe_X(x)\neq \spe_X(x')$ then, swapping $x$ and $x'$ if necessary, we may assume that
 $\spe_X(x')$ does not lie in the closure of $\{\spe_X(x)\}$.
 Then we can separate $x$ and $x'$ by the model function $\phi_{Y,D}$ where $Y$ is the blow-up of $X$ along the closure of $\{\spe_X(x)\}$ (with its reduced induced structure) and $D$ is the inverse image of the center of the blow-up. Indeed, $\phi_{Y,D}(x)<1$ whereas $\phi_{Y,D}(x')=1$. If $\spe_X(x)= \spe_X(x')$ then we can find an open subscheme $U$ of $X$ and a regular function $f$ on $U$ such that $\spe_X(x)$ lies in $U$ and
 $|f(x')|\neq |f(x)|$. We denote by $D_0$ the schematic closure in $X$ of the closed subscheme of $U$ defined by $f=0$. Let $Y$ be the blow-up of $X$ along $D_0$ and denote by $D$ the inverse image of $D_0$ on $Y$. Then $\phi_{Y,D}(x)=|f(x)|$ and $\phi_{Y,D}(x')=|f(x')|$, so that $\phi_{(X,D)}$ separates $x$ and $x'$.

 The Stone-Weierstrass theorem now implies that $\mathcal{T}(X)$ is dense in
  $\mathcal{C}^0(X^{\an},\R)$ for the topology of uniform convergence. Thus it suffices to show that for every model function $\phi=\phi_{Y,D}$ of $X$, the sum
   $$q^{-n[F'\colon F]}\sum_{x\in X(F')}\phi(x)$$ converges to $c_X\phi(\eta)=c_X$
  as $F'$ ranges through $\mathcal{E}^a_F$ (here we made a small abuse of notation by writing $\phi(x)$ for the value of $\phi$ at the image of $x$ in $X^{\an}$).
By definition, $\phi(x)=0$ if $x$ lies in $D(F')$, and $\phi(x)=1$ otherwise.
  Thus, writing $Z=Y\setminus D$, we have
  $$q^{-n[F'\colon F]}\sum_{x\in X(F')}\phi(x)=q^{-n[F'\colon F]}|Z(F')|$$
  and this expression converges to $c_Z=c_X$ by the Lang--Weil estimates for $Z$.
\end{proof}

\begin{remark}
One can also formulate a variant of Proposition \ref{prop:langweil} that involves pluricanonical forms on $X$, to strengthen the analogy with the local field case that will be discussed below. Let $X$ be a connected smooth and proper $F$-scheme, and let $\theta$ be a non-zero pluricanonical form on $X$. For every finite extension $F'$ of $F$, this form induces a discrete measure on $X(F')$ that gives mass zero to every point in the zero locus $\mathrm{div}(\form)$ of $\theta$, and mass one to every other point in $X(F')$.
 We denote by $\nu_{\form,F'}$
the pushforward to $X^{\mathrm{an}}$ of the normalized measure
 $$q^{-n[F'\colon F]}\sum_{x\in (X\setminus \mathrm{div}(\theta))(F')}\delta_x$$ on $X(F')$.
 The same arguments as in the proof of Proposition \ref{prop:langweil} shows that this measure  converges to $c_X\delta_{\eta}$ as $F'$ ranges through $\mathcal{E}_F^a$.
\end{remark}

\subsection{Model functions}
Now we return to our original setting, where $K$ is a complete discretely valued field.
 Let $X$ be a smooth and proper $K$-scheme of pure dimension $n$.

In order to study convergence of measures on $X^{\an}$, we again need a suitable class of test functions.
Let $\cX$ be a normal proper $R$-model for $X$ and let $D$ be an effective Cartier divisor on $\cX$. Then $D$ defines a continuous function on $X^{\an}$:
 $$\phi_D:X^{\an}\to [0,1]:x\mapsto |f(x)|$$ where $f$ is a local equation for $D$ at the point $\spe_{\cX}(x)$. We call such a function a {\em model function} for $X$.
  We call $D$ a {\em vertical divisor} if it is supported on $\cX_k$;
 in that case, we also say that the model function $\phi_D$ is vertical (this is equivalent to the property that $\phi_D$ is nowhere vanishing).
 Note that $\phi_D$ remains invariant under pullback of $D$ to a proper $R$-model dominating $\cX$, and that $\phi_{D+D'}=\phi_D\cdot \phi_{D'}$ for all
 effective Cartier divisors supported on $\cX_k$. It follows that the class of vertical model functions for $X$ is closed under multiplication.

 \begin{prop}\label{prop:test}
 For every pair of distinct points $x$, $x'$ on $X^{\an}$, we can find a proper $R$-model $\cX$ and a vertical effective Cartier divisor $D$ on $\cX$ such that $\phi_D(x)\neq \phi_D(x')$.
 \end{prop}
 \begin{proof}
 It is well-known that we can find a proper $R$-model $\cY$ of $X$ such that $\spe_{\cY}(x)\neq \spe_{\cY}(x')$; see for instance~\cite[2.3.2]{MuNi}.
  By symmetry, we may assume that $\spe_{\cY}(x')$ is not contained in the Zariski-closure of $\{\spe_{\cY}(x)\}$. Denote by $Z$ the schematic closure of $\{\spe_{\cY}(x)\}$ in $\cY$.  Then we can take for $\cX$ the blow-up of $\cY$ at $Z$, and for $D$ the schematic inverse image of $Z$ in $\cX$. Indeed, $\phi_D(x')=1$ whereas $\phi_D(x)<1$.
 \end{proof}
\begin{cor}\label{cor:stone}
 The real sub-vector space of $\mathcal{C}^0(X^{\an},\R)$ generated by the vertical model functions
 is dense in the topology of uniform convergence.
\end{cor}
\begin{proof}
The class of vertical model functions is closed under multiplication, and it separates points by
 Proposition~\ref{prop:test}. It also contains the constant function $\phi_{\cX_k}$ with value $|\unif|$.
  Thus the result follows from the Stone-Weierstrass theorem.
\end{proof}

In our calculations below, we will need the following useful property.

\begin{prop}\label{prop:adapted}
Let $\cX$ be a log regular proper $R$-model of $X$. Let $\phi$ be a vertical model function on $X^{\an}$. Then there exist a toroidal resolution $\cY\to \cX$ and a closed subset $A$ of $\cY_k$ such that $A$ does not contain any strata of $\cY_k$ and $\phi=\phi\circ \rho_{\cY}$ on $X^{\an}\setminus \spe^{-1}_{\cY}(A)$.
\end{prop}
\begin{proof}
 We say that a coherent ideal sheaf $\mathscr{I}$ on $\cX$ is vertical if it contains $\unif^\ell$ for some $\ell>0$. We can then define a function
 $$\phi_{\mathscr{I}}\colon X^{\an}\to (0,1]$$ by taking $\phi_{\mathscr{I}}(x)$ to be the maximum of the absolute values $|f(x)|$ where $f$ runs through the stalk of $\mathscr{I}$ at $\spe_{\cX}(x)$. When $\mathscr{I}$ is the defining ideal sheaf of a vertical effective divisor $D$, we have $\phi_{\mathscr{I}}=\phi_D$.

 It follows from Proposition 2.2 in~\cite{BFJ16} that we can find vertical ideal sheaves $\mathscr{I}$ and $\mathscr{J}$ on $\cX$ such that $\phi=\phi_{\mathscr{I}}/\phi_{\mathscr{J}}$ (the overall assumption in~\cite{BFJ16} that $k$ has characteristic zero was not used in the proof of that proposition). Thus it suffices to prove the statement for $\phi=\phi_{\mathscr{I}}$, where $\mathscr{I}$ is any vertical ideal sheaf on $\cX$.

 Replacing $\cX$ by a toroidal resolution, we may assume that $\cX$ is an $snc$-model of $X$.
 On each face of the skeleton $\Sk(\cX)$, the function $\log(\phi)$ is the maximum of finitely many piecewise $K$-affine functions, by~\cite[3.2.2]{MuNi}. Thus $\log(\phi)$ is itself piecewise $K$-affine. Subdividing the skeleton $\Sk(\cX)$ and taking the corresponding toroidal modification of $\cX$, we can reduce to the case where $\log(\phi)$ is affine on every face of $\Sk(\cX)$. By means of a further subdivision, we can arrange that $\cX$ is still an $snc$-model.

 Now we decompose $\mathscr{I}$ as $\mathscr{J}\cdot \mathcal{O}_{\cX}(-D)$ where $D$ is a vertical effective divisor on $\cX$ and $\mathscr{J}$ is a vertical ideal sheaf
  satisfying $\mathscr{J}_{\xi}=\mathcal{O}_{\cX,\xi}$ for every generic point $\xi$ of $\cX_k$. Then $\phi=\phi_{\mathscr{J}}\cdot \phi_D$, and it follows directly from the definition of the retraction map $\rho_{\cX}$ that $\phi_D=\phi_D\circ \rho_{\cX}$. Let $A$ be the zero locus of the ideal sheaf $\mathscr{J}$ on $\cX$. Since $\phi_{\mathscr{J}}=1$, and thus $\phi=\phi_D$, on the complement of $\spe^{-1}_{\cX}(A)$, it is enough to show that $A$ does not contain any stratum of $\cX_k$.

 The functions $\log(\phi)$ and $\log(\phi_D)$ are affine on every face of $\Sk(\cX)$, so that the same holds for $\phi_{\mathscr{J}}$. Since  $\phi_{\mathscr{J}}$ takes the value $1$ at every vertex of $\Sk(\cX)$, this implies that $\phi_{\mathscr{J}}=1$ on the whole of $\Sk(\cX)$. This  means that at the generic point $\xi$ of every stratum of $\cX_k$, we have $\mathscr{J}_{\xi}=\mathcal{O}_{\cX,\xi}$ (otherwise, $\phi_{\mathscr{J}}$ would be strictly smaller than $1$ on the interior of the face of $\Sk(\cX)$ corresponding to this stratum). Thus $A$ does not contain any stratum of $\cX_k$.
\end{proof}

\subsection{The measure associated with a pluricanonical form}\label{ss:pluri}
For the remainder of this section, we assume that $k$ is finite. We denote its characteristic by $p$ and its cardinality by $q=p^a$. We normalize the absolute value on $K$ by $|\unif|_K=q^{-1}$, where $\unif\in R$ is any uniformizer.
The valued field $K$ is either a finite extension of $\Q_p$, or a Laurent series field over the finite field $k$. In both cases, $K$ is locally compact,
 and thus carries a Haar measure $\mu_{\mathrm{Haar}}$, which we normalize by requiring that $\mu_{\mathrm{Haar}}(R)=1$.

As $X$ is a smooth and proper $K$-scheme $X$ of pure dimension $n$, the set $X(K)$ of rational points  can be turned into a compact $n$-dimensional $K$-analytic manifold in a canonical way.
  Moreover, every canonical form $\form$ on $X$ gives rise to a canonical form on $X(K)$, which at its turn induces a measure $|\form|$ on $X(K)$; see for instance~\cite[\S1.1]{ChNiSe}. This construction can be immediately generalized to pluricanonical forms; we will briefly review the definitions and basic properties. Let $\form$ be an $m$-canonical form on $X$, for some positive integer $m$. We can partition $X(K)$ into finitely many compact open sets $U_1,\ldots,U_r$ such that, for every $i$ in $\{1,\ldots,r\}$, there exists a $K$-analytic open embedding
  $\phi_i:U_i\to K^n$. We set $V_i=\phi_i(U_i)$ and we denote by $\psi_i:V_i\to U_i$ the inverse of $\phi_i$. Then we can write $\psi_i^* \form$ as
  $g_i dx^{\otimes m}$ for a unique $K$-analytic function $g_i$ on $V_i$, where $dx$ denotes the standard volume form on $K^n$.
    The measure $|\form|$ is now characterized by the property that, for every continuous function $f:X(K)\to \R$, we have
  $$\int_{X(K)} f |\form|=\sum_{i=1}^r \int_{V_i}(f\circ \psi_i)|g_i|^{1/m}_K d\mu_{\mathrm{Haar}}.$$  The change of variables formula guarantees that this definition does not depend on the choice of the charts $\phi_i$.

\subsection{Convergence of the measure under unramified extensions}
Now let
$\form$ be an $m$-canonical form on $X$, for some $m>0$.
Assume that $\form$ is not identically zero on any connected component of $X$.

For any finite extension $K'/K$, we have a continuous map
\begin{equation*}
  \pi_{K'}\colon X(K')\to X^{\an},
\end{equation*}
defined as the composition of the injection $X(K')\hookrightarrow(X')^{\an}$ and the projection $\pr\colon (X')^{\an}\to X^{\an}$.

Recall from Section~\ref{sec:shilov} that we denote by $\ordmin(X,\form)$ the minimum of
  $\ord_C(\form)$ over the connected components $C$ of some weak N\'eron model $\cU$ of $X$; this value does not depend on $\cU$, by Proposition~\ref{prop:indep}.


 \begin{definition}\label{def:unram-measure}
   For every finite unramified extension $K'$ of $K$, we set
   \begin{equation*}
     \nu_{\theta,K'}:=q^{\ordmin(X,\form)[K'\colon K]}(\pi_{K'})_*|\form\otimes_K K'|;
   \end{equation*}
   this is a positive Radon measure on $(X')^{\an}$.
 \end{definition}

   Let $\mathcal{E}^{\ur}_K$ be the set of finite extensions $K'$ of $K$ in $K^{\ur}$, ordered by inclusion.
Our next result refines the unramified case of the Main Theorem in the introduction. It is valid unconditionally in all characteristics; in particular, it does not require the existence of an snc-model of $X$.
 \begin{theorem}\label{thm:shi}
  As $K'$ ranges through $\mathcal{E}^{\ur}_K$, the measures $\nu_{\form,K'}$ converge to the discrete measure
 $$\mu_{X,\form}=\sum_{x\in \Sh(X,\form)}\tdeg_{K}(x)\delta_x$$ on $X^{\an}$, where $\tdeg_K$ is the tame degree function from Section \ref{ss:tamedeg}, and $\Sh(X,\form)$ is the Shilov boundary of $(X,\form)$ defined in Section \ref{sec:shilov}.
\end{theorem}
 If $\cU$ is a smooth $R$-model of $X$,  $C$ is a connected component of $\cU_k$, and $x\in X^{\an}$ is the divisorial point associated with  $(\cU,C)$, then
 $\mathscr{H}(x)$ is a discretely valued field of ramification index one over $K$ with residue field $k(C)$. Since $C$ is smooth over $k$, the field $k(C)$ is separable over $k$.
 It follows that the tame closure of $K$ in $\mathscr{H}(x)$ is algebraically closed, and that
  $\tdeg_K(x)$ is equal to the number of geometric connected components of $C$. In particular, it is finite.

The proof of Theorem~\ref{thm:shi} relies on the following estimates.
\begin{prop}\label{prop:int}
  Let $\cU$ be a smooth $R$-model of $X$, and write
  $$\mathrm{div}_{\cU}(\form)=\sum_{C\in \pi_0(\cU_k)}\ord_C(\form)C+H$$ where $H$ is an effective Cartier divisor on $\cU$ that is horizontal (i.e., flat over $R$).
   Let $D$ be a vertical effective Cartier divisor on $\cU$. For every connected component $C$ of $\cU_k$, we denote by $\ord_C(D)$ the multiplicity of $C$ in $D$.

Let $K'$ be a finite unramified extension of $K$, with valuation ring $R'$ and residue field $k'$.
Then we have
\begin{multline*}
  \sum_{C\in \pi_0(\cU_k)}q^{-(\ord_C(\form)+\ord_C(D))[K'\colon K]}|(C\setminus H)(k')|\\
  \leq q^{\dim(X)[K'\colon K]} \int_{\cU(R')}\phi_D\circ\pi_{K'} |\form\otimes K'| \\
  \leq \sum_{C\in \pi_0(\cU_k)}q^{-(\ord_C(\form)+\ord_C(D))[K'\colon K]}|C(k')|.
\end{multline*}
In particular, if $\cU$ is a weak N\'eron model of $X$, then
\begin{multline*}
  \sum_{C\in \pi_0(\cU_k)}q^{-(\ord_C(\form)+\ord_C(D))[K'\colon K]}|(C\setminus H)(k')|\\
  \leq q^{\dim(X)[K'\colon K]} \int_{X(K')}\phi_D\circ\pi_{K'} |\form\otimes K'| \\
\leq \sum_{C\in \pi_0(\cU_k)}q^{-(\ord_C(\form)+\ord_C(D))[K'\colon K]}|C(k')|.
\end{multline*}
\end{prop}
  \begin{proof}
  Since the property of being a weak N\'eron model is preserved by unramified base change, and $q^{[K'\colon K]}$ is the cardinality of the residue field of $K'$, we may assume that $K=K'$.
   Let $C$ be a connected component of $\cU_k$, and let $\xi$ be a point of $C(k)$.
    We set $A=X(K)\cap \spe_{\cU}^{-1}(\xi)$.
   Then it suffices to show that
   $$\int_{A}\phi_D|\form| \leq q^{-\dim(X)-\ord_C(\form)-\ord_C(D)},$$
   and that equality holds when $\xi$ is not contained in $H$. Multiplying $\form$ with $\unif^{-\ord_C(\form)m}$, we can reduce to the case where
   $\ord_C(\form)=0$. Since $\phi_D$ is constant on $A$ with value $q^{-\ord_C(D)}$, we may also assume that $D=0$, and hence $\phi=1$.

 Let $h$ be a local equation for the Cartier divisor $H$ at the point $\xi$.
We denote by $\mathfrak{m}$ the maximal ideal in $R$.
  Locally at $\xi$, we can find an \'etale morphism of $R$-schemes $\cU\to \A^n_R$ that maps $\xi$ to the origin of $\A^n_k$.
  This morphism induces a $K$-analytic isomorphism $\phi:A\to \mathfrak{m}^n$.  Moreover,
   the form induced by $\form$ on $S$ can be written as $uh\phi^*(dx)^{\otimes m}$ where $dx$ is the standard volume form on $K^n$ and $u$ is a unit in $\mathcal{O}_{\cU,\xi}$.
We have $|u|=1$ and $|h|\leq 1$ on $S$, and $|h|=1$ if $\xi$ is not contained in $H$. Now the result follows from the fact that the Haar measure of $\mathfrak{m}^n$ is equal to  $q^{-\dim(X)}$.
  \end{proof}

\begin{proof}[Proof of Theorem~\ref{thm:shi}]
 By Corollary~\ref{cor:stone},  it is enough to show that
$$\int_{X^{\an}}\phi_D\, d\nu_{\form,K'}=\int_{X(K')}\phi_D\circ \pi_{K'}|\form\otimes K'|$$ converges to $$\sum_{x\in \Sh(X,\form)}\deg_K(x)\phi_D(x)$$ for
 every proper $R$-model $\cX$ of $X$ and every effective Cartier divisor $D$ supported on $\cX_k$.
  Applying a N\'eron smoothening to $\cX$ and pulling back $D$, we may assume that the $R$-smooth locus $\cU=\Sm(\cX)$ of $\cX$ is a weak N\'eron model of $X$.
  Since the property of being a weak N\'eron model is preserved by unramified base change, the measures $\nu_{\form,K'}$ are all supported on the compact analytic domain $\widehat{\cU}_{\eta}$ in $X^{\an}$. For every connected component $C$ of $\cU_k$, we denote by $\ord_C(D)$ the multiplicity of $D$ along $C$. Then
 the function $\phi_D$ is constant on $\spe_{\cX}^{-1}(C)$ with value $q^{-\ord_C(D)}$. The result now follows from Proposition~\ref{prop:int} and the Lang--Weil estimates:
 if we denote by $x\in X^{\an}$ the divisorial point associated with $(\cU,C)$, then
 we have
 \begin{equation*}
   |C(k')|
   \sim |\pi_0(C\otimes_k k^s)|q^{\mathrm{dim}(X)[K'\colon K]}
   =    \tdeg_K(x)q^{\mathrm{dim}(X)[K'\colon K]}
 \end{equation*}
 as $[K':K]\to\infty$, and the same holds for $|(C\setminus H)(k')|$.
\end{proof}

\subsection{A convergence lemma}
In order to study convergence of pluricanonical measures under ramified ground field extension, we will use the following result.
\begin{lemma}\label{lemm:int}
Let $P$ be a convex polytope in $\R^n$ with $\Q$-rational vertices. Let $\phi\colon P\to \R$ be a continuous function, and let $$\alpha\colon \R^{n}\to \R$$ be an affine function.
Assume that $\alpha$ is positive on the relative interior $\mathring{P}$ of $P$, and denote by $d$ the dimension of the maximal face $\tau$ of $P$ on which $\alpha$ vanishes; if $\tau$ is empty, we set $d=-1$.
We write $\lambda_{\tau}$ for the integral Lebesgue measure on the affine space spanned by $\tau$.
 We fix a real number $r>1$. For all positive integers $e$ and $f$, we set
$$S(e,f)=e^{-d}\sum_{x\in P\cap (1/e)\Z^n}\phi(x)r^{-ef\alpha(x)}.$$

\begin{enumerate}
\item \label{it:convwild} The value $S(e,f)$ converges to $$I_{\tau}(\phi)=\int_{\tau}\phi\, d\lambda_{\tau}$$ as $e$ and $f$ run through the positive integers ordered by divisibility. More explicitly, for every $\varepsilon>0$, there exist positive integers $e_0$ and $f_0$ such that
$$|S(e,f)-I_{\tau}(\phi)|<\varepsilon$$ whenever $e_0|e$ and $f_0|f$.

\item \label{it:convtame} Let $p$ be a prime number. Denote by $\N'$ the set of positive integers that are prime to $p$, ordered by divisibility. If $\tau$ has a $\Z_{(p)}$-integral point, then
 $S(e,f)$ converges to $I_{\tau}(\phi)$ as $e$ and $f$ run through the directed set $\N'$. If $\tau$ does not have a $\Z_{(p)}$-integral point, then $S(e,f)$ converges to $0$.
 \end{enumerate}
\end{lemma}
\begin{proof}
 We can reduce to the case where $\phi$ is nonnegative. Let $M$ be the maximum of $\phi$ on $P$.
 If $\tau $ is empty, we can bound $r^{-ef\alpha(x)}$ from above on $P$ by $c^e$ for some positive constant $c<1$; thus $S(e,f)$ converges to $0$ as $e\to \infty$, uniformly in $f$. Therefore, we may assume that $\tau$ is non-empty.

  For all positive integers $e$, we set
  $$S_{\tau}(e)=e^{-d}\sum_{\tau\cap (1/e)\Z^n}\phi(x).$$
 Since $\alpha$ vanishes on $\tau$, we have, for all $e$ and $f$,
 $$S_{\tau}(e) \leq S(e,f)\leq S_{\tau}(e)+ Me^{-d}\sum_{(P\setminus \tau)\cap (1/e)\Z^n}r^{-ef\alpha(x)}.$$
 The expression $S_{\tau}(e)$ converges to the integral $I_{\tau}(\phi)$ as $e$ runs through the positive integers ordered by divisibility.  If $\tau$ does not have a $\Z_{(p)}$-integral point,
 then $S_{\tau}(e)=0$ for every $e$ in $\N'$.
 If $\tau$ has a $\Z_{(p)}$-rational point, then we can find an element $e$ in $\N'$
and an integral $(1/e)\Z$-affine isomorphism between $\tau$ and a $d$-dimensional polytope in $\R^d$.
 It follows that $S_{\tau}(e)$ converges to $I_{\tau}(\phi)$ as $e$ runs through the directed set $\N'$.

Thus it suffices to show that  the expression
\begin{equation}\label{eq:zerosum}
e^{-d}\sum_{x\in (P\setminus \tau)\cap (1/e)\Z^n}r^{-ef\alpha(x)}
\end{equation}
 lies arbitrarily close to $0$ when $e$ and $f$ are sufficiently large.
  After a $\Z$-linear coordinate transformation on $\R^n$, we may assume that the affine span of $\tau$ is the affine subspace defined by $$x_{d+1}=q_{d+1},\ldots,x_n=q_n$$
 where $q_{d+1},\ldots,q_n$ are rational numbers. Subdividing $P$ into finitely many polytopes, we may also assume that $P$ is contained in the domain of $\R^n$ defined by
 $$x_1\geq 0,\,\ldots,\,x_d\geq 0,\,x_{d+1}\geq q_{d+1},\,\ldots,\,x_n\geq q_n.$$
 Since $\alpha$ vanishes on $\tau$ and is positive on the relative interior of $P$, it is of the form
 $$\alpha(x_1,\ldots,x_n)=a_{d+1}(x_{d+1}-q_{d+1})+\ldots+ a_{n}(x_{n}-q_{n})$$ for some nonnegative real numbers $a_{d+1},\ldots,a_n$, not all zero.
 In fact, we may assume that $a_j>0$ for all $j>d$. Indeed, we can pick $b_j>0$, $d<j\le n$, such that if we set $\beta(x_1,\dots,x_n)=\sum_{d+1}^nb_j(x_j-q_j)$, then $\beta\le\alpha$ on $P$ (it suffices to verify this inequality on the vertices of $P$), and it then suffices to show that the expression in~\eqref{eq:zerosum} tends to zero when $\alpha$ is replaced by $\beta$.

 We choose a positive integer $N$ such that
$P\setminus \tau$ is contained in $$Q=[0,N]^{d}\times \left(\,([q_{d+1},N]\times \cdots \times [q_{n},N])\setminus \{(q_{d+1},\ldots,q_n)\}\,\right).$$ Then we can bound~\eqref{eq:zerosum} from above by
\begin{equation}\label{eq:zerosum2}
e^{-d}\sum_{x\in Q\cap (1/e)\Z^n}r^{-ef\alpha(x)}.
\end{equation}
 This sum can be computed explicitly.  For every $i$ in $\{d+1,\ldots,n\}$, we denote by $m_i$ the smallest integer such that
 $m_i\geq eq_i$. We set $\delta_{e,q}=1$ if $eq_i$ is an integer for every $i$ in $\{d+1,\ldots,n\}$, and $\delta_{e,q}=0$ otherwise. Then
  $$\eqref{eq:zerosum2}=\frac{(1+eN)^d}{e^d}\left( \prod_{i=d+1}^{n}\frac{r^{-fa_i(eN-eq_i+1)}- r^{-fa_i(m_i-eq_i)}}{r^{-fa_i}-1} - \delta_{e,q}\right).$$
 By definition, $\delta_{e,q}=1$ if and only if $m_i-eq_i=0$ for all $i$; otherwise, $\delta_{e,q}=0$, and $m_i-eq_i$ is nonnegative for all $i$ and positive for some $i$. Moreover, if $m_i-eq_i$ is positive and we write $q_i=u_i/v_i$ with $u_i$, $v_i$ integers and $v_i>0$, then $m_i-eq_i$ is bounded below by $1/v_i$, which is independent of $e$.
 It follows that
  $$ \prod_{i=d+1}^{n}\frac{r^{-fa_i(Ne+1-eq_i)}- r^{-fa_i(m_i-eq_i)}}{r^{-fa_i}-1} - \delta_{e,q} \to 0$$ as $f\to \infty$, uniformly in $e$. This concludes the proof.
\end{proof}

\subsection{Convergence of the measure under ramified extensions}
 Let  $\mathcal{E}^a_K$ be the set of finite extensions $K'$ of $K$ in the algebraic closure $K^a$, ordered by inclusion. We denote by
   $\mathcal{E}^t_K$ the subset of $\mathcal{E}^a_K$ consisting of tamely ramified extensions.
   For every finite extension $K'$ of $K$, we denote by $e(K'/K)$ its ramification degree over $K$, and by $f(K'/K)$ the degree of the residue field of $K'$ over the residue field $k$ of $K$. Since $K$ is complete, we have
   $$[K'\colon K]=e(K'/K)\cdot f(K'/K).$$
   As above, $X$ is a smooth and proper $K$-scheme of pure dimension, and $\form$ a non-zero $m$-canonical form on $X$, for some $m>0$, such that
  $\form$ is not identically zero on any connected component of $X$.

  The following definition is a version of Definition~\ref{def:unram-measure} for ramified extensions.
 \begin{definition}\label{def:ram-measure}
   For every finite extension $K'$ of $K$, we set
   \begin{equation*}
     \nu_{\form,K'}:=\frac{q^{\wt_{\min}(X,\form)[K' \colon K]}}{e(K'/K)^{\dim(\Sk(X,\form))}}
     (\pi_{K'})_*|\form\otimes_K K'|
   \end{equation*}
   and
   \begin{equation*}
     \nu^t_{\form,K'}:=\frac{q^{\wt_{\min}(X,\form)[K' \colon K]}}{e(K'/K)^{\dim(\Sk^t(X,\form))}}
     (\pi_{K'})_*|\form\otimes_K K'|,
   \end{equation*}
 where $\pi_{K'}\colon X(K')\to X^{\an}$ is the canonical map and $\Sk^t(X,\form)$ is the temperate part of $\Sk(X,\form)$ (see Definition \ref{defi:skmeasures}). If $\Sk^t(X,\form)$ is empty, we set $\dim(\Sk^t(X,\form))=-1$.
 \end{definition}
 Thus $\nu_{\form,K'}$ and $\nu^t_{\form,K'}$ are positive Radon measures on $X^{\an}$.
\begin{theorem}\label{thm:padic}
Assume that $X$ has an $snc$-model.
\begin{enumerate}
\item \label{it:maintame} As $K'$ ranges through $\mathcal{E}^t_K$, the measure $\nu^t_{\form,K'}$
 converges to the stable Lebesgue measure $\Leb^{\st}_{\Sk^t(X,\form)}$
 supported on the temperate part of the Kontsevich--Soibelman skeleton of the pair $(X,\form)$.
\item \label{it:mainwild} If $X$ has a log smooth proper $R$-model, then as $K'$ ranges through $\mathcal{E}^a_K$, the measure $\nu_{\form,K'}$
 converges to the stable Lebesgue measure $\Leb^{\st}_{\Sk(X,\form)}$
 supported on the Kontsevich--Soibelman skeleton of the pair $(X,\form)$.
 \end{enumerate}
\end{theorem}
Together with Theorem~\ref{thm:shi}, this proves the Main Theorem in the introduction.
\begin{proof}
\eqref{it:maintame}
 Let $\cX$ be a log regular proper $R$-model of $X$.
 For every point $x$ of $\Sk(\cX)$, we denote by $E_x$ the unique stratum of $\cX_k$ with generic point $\spe_{\cX}(x)$; this is the stratum corresponding to the unique face of $\Sk(\cX)$ whose relative interior contains $x$. We denote by $E_x^o$ the complement in $E_x$ of the intersections with all the strata in $\cX_k$ that do not contain $E_x$.
  As $x$ runs through $\Sk(\cX)$, the sets $E_x^o$ form a partition of $\cX_k$ into locally closed subsets.

Let $\phi=\phi_D$ be a vertical model function on $X^{\an}$. By Proposition~\ref{prop:adapted}, we may assume that there exists a closed subset $A$ of $\cX_k$ that does not contain any stratum of $\cX_k$ and such that
$\phi=\phi\circ \rho_{\cX}$ on the complement of $\spe^{-1}_{\cX}(A)$.
  Let $L$ be a tame finite extension of $K$. It suffices to prove the theorem for the pair $(X\otimes_K L,\form\otimes_K L)$ over $L$ instead of $(X,\form)$: on the one hand,
 $$\wt_{\min}(X\otimes_K L,\form\otimes_K L)=e(L/K)\wt_{\min}(X,\form)$$ and $\dim(\Sk^t(X\otimes_K L,\form\otimes_K L)=\dim(\Sk^t(X,\form))$, by Proposition~\ref{prop:skbasech} and the definition of the temperate part of the skeleton; on the other hand,
$\Leb^{\st}_{\Sk^t(X,\form)}$ is the pushforward of $$\frac{1}{e(L/K)^{\dim(\Sk^t(X,\form))}}\Leb^{\st}_{\Sk^t(X\otimes_K L,\form\otimes_K L)}$$ to $X^{\an}$, by Propositions~~\ref{prop:skbasech} and \ref{prop:stmeasure2}. Thus, replacing $\cX$ by its normalized base change to the valuation ring of a suitable tame finite extension of $K$, and $A$ by its inverse image, we can reduce to the case where the multiplicities of the components of $\cX_k$ are powers of $p$, and all the strata of $\cX_k$ are geometrically connected, by Corollary~\ref{coro:ppower}.

  We denote by $H$ the Zariski closure in $\cX$ of the zero locus of $\form$ on $X$ (this is the support of the horizontal part of the divisor $\mathrm{div}_{\cX}(\form)$).
 Let $K'$ be a tame finite extension of $K$, and set $e=e(K'/K)$ and $f=f(K'/K)$; then $[K'\colon K]=ef$. We denote by $R'$ the integral closure of $R$ in $K'$, and by $k'$ its residue field.
 We claim that the following inequalities hold:
\begin{align*}
 & \left(\frac{q^f-1}{q^f}\right)^{\dim(X)} \sum_{x\in \Sk(X,\form)((1/e)\Z)}\phi(x)q^{-\wt_{\form}(x)ef-\dim(E_x)f}|(E_x^o\setminus (A\cup H))(k')|
 \\ & \leq \int_{X(K')}\phi\circ\pi_{K'}\,|\form\otimes_K K'|
 \\
  & \leq \sum_{x\in \Sk(\cX)((1/e)\Z)} \phi(x)q^{-\wt_{\form}(x)ef-\dim(E_x)f}|E_x^o(k')|.
   \end{align*}
   Proposition~\ref{prop:sklogmodel} implies that
    $E^o_x\setminus H$, and thus $E^o_x\setminus (A\cup H)$,  are non-empty whenever $x$ lies in $\Sk(X,\form)$.
 By the Lang--Weil estimates, $q^{-\dim(E_x)f}|E_x^o(k')|$ tends to $1$ as $f\to \infty$, and the same holds for $q^{-\dim(E_x)f}|(E_x^o\setminus (A\cup H))(k')|$ if $E_x^o\setminus H$ is non-empty. The weight function $\wt_{\form}$ is piecewise rational affine on $\Sk(\cX)$, and achieves  its minimal value $\wt_{\min}(X,\form)$ precisely on $\Sk(X,\form)$.
  Now Propositions~\ref{prop:tame} and~\ref{prop:tamecrit}, together with Lemma~\ref{lemm:int}, imply  that
$$\frac{q^{\wt_{\min}(X,\form)ef}}{e^{\dim(\Sk^t(X,\form))}}\int_{X(K')}\phi\circ\pi_{K'}\,|\form\otimes_K K'|$$
lies arbitrarily close to
$$\int_{\Sk^t(X,\form)}\phi\,d\Leb^{\st}_{\Sk^t(X,\form)}$$ when $K'$ is sufficiently large; this is what we wanted to prove.

 So let us prove our claim.  We write $X'=X\otimes_K K'$ and we denote by $\form'=\form \otimes_K K'$ the pullback of $\form$ to $X'$.
 Let $\cX'$ be the normalization of $\cX\otimes_R R'$.  Then the projection morphism $$\mathrm{pr}\colon (X')^{\an}\to X^{\an}$$ induces a bijection
 $$\Sk(\cX')(\Z)\to \Sk(\cX)((1/e)\Z),$$ by Proposition~\ref{prop:compat2}. For every point $x'$ in $\Sk(\cX')$, we can define a locally closed subset $E_{x'}^o$ in the special fiber of $\cX'$, in the same way as for the model $\cX$.
 By the same calculation as in the proof of Proposition~4.1.2 in \cite{logzeta}, the morphism $\cX'\to \cX$
 maps $E^o_{x'}$ isomorphically onto $E_{\mathrm{pr}(x')}^o\otimes_k k'$. Since the residue field of $K'$ has cardinality $q^f$, and
 we have $\wt_{\form'}=e\cdot \wt_{\form}\circ \mathrm{pr}$ by Proposition~\ref{prop:skbasech}, it is enough to prove the inequalities for $K'=K$; then $e=f=1$.

Let $h\colon \cY\to \cX$ be a toroidal resolution of $\cX$, associated with a regular proper subdivision of the fan $F(\cX)$.
  Then $\Sk(\cY)$ is a subdivision of the polyhedral complex $\Sk(\cX)$; both are equal as subspaces of $X^{\an}$, and they have the same piecewise $\Z$-integral structure (in particular, the same set of $\Z$-integral points).
   For every $y$ in $\Sk(\cY)$, we define a locally closed subset $F^o_y$ in $\cY_k$ in the same way as above (we use the letter $F$ to distinguish it from the members of the partition $\{E^o_x\}$ of $\cX_k$).
 Since $\cY$ is a regular proper $R$-model for $X$, its $R$-smooth locus is a weak N\'eron model for $X$, by~\cite[3.1/2]{BLR}; we denote it by $\cU$.
 The connected components of $\cU_k$ are exactly the sets $F_y^o$ for $\Z$-integral points $y$ on $\Sk(\cY)$ (note that all the $\Z$-integral points of $\Sk(\cY)$ are vertices of $\Sk(\cY)$, by regularity of the fan $F(\cY)$; moreover, a vertex of $\Sk(\cY)$ is $\Z$-integral if and only if the corresponding component of $\cY_k$ has multiplicity one).
 The morphism $h$ induces a morphism of $k$-schemes $F^o_y\to E^o_y$, and it follows from  the construction of the morphism $h$ in the
proof of~\cite[9.9]{kato} that $F^o_y$ is a torsor over $E^o_y$ with translation group
 $\mathbb{G}^{\dim(X)-\dim(E_x^o)}_{m,k}$. In particular, the map $F^o_y(k)\to E^o_y(k)$ is surjective, and its fibers all have cardinality $(q-1)^{\dim(X)-\dim(E_x^o)}$.

 We can use the weak N\'eron model $\cU$ to get the bounds we need for the integral $\int_{X(K)}\phi|\form|$.
  In fact, we will prove a stronger property: for every $y$ in $\Sk(\cY)(\Z)$, we denote by $\cU_y$ the open subscheme of $\cU$
 obtained by deleting all the components in the special fiber except for $F^o_y$. Since $\cU$ is a weak N\'eron model for $X$, the set $X(K)$ is the union of the compact open subsets $\cU_y(R)$ over all the points $y$ in  $\Sk(\cY)(\Z)$. We fix such a point $y$; then it is enough to prove that
  \begin{align*}
 &  \left(\frac{q-1}{q}\right)^{\dim(X)} \phi(y) q^{-\wt_{\form}(y)-\dim(E_y)}|(E_y^o\setminus (A\cup H))(k)|
 \\ & \leq \int_{\cU_y(R)}\phi|\form|
 \\
  & \leq \phi(y)q^{-\wt_{\form}(y)-\dim(E_y)}|E_y^o(k)|
   \end{align*}

 We first deduce the upper bound. It follows easily from the definition of the vertical model function $\phi=\phi_D$ that,
  for every $y$ in $\Sk(\cY)(\Z)$, the restriction of $\phi$ to $\cU_y(R)$ is bounded above by $\phi(y)$.
  Thus we may assume that $D=0$ and $\phi=1$; then Proposition~\ref{prop:int} tells us that
  $$\int_{\cU_y(R)}|\form|\leq q^{-\wt_{\form}(y)-\dim(X)}|F^o_y(k)|= q^{-\wt_{\form}(y)-\dim(E^o_y)}|E^o_y(k)|.$$
 Finally, we prove the lower bound. We have constructed the set $A$ in such a way that $\phi(u)=\phi(y)$ for all the points
 $u$ in $(\cU_y\setminus h^{-1}(A))(R)$. Thus, we may again assume that $D=0$ and $\phi=1$. Since the horizontal part of $\mathrm{div}_{\cY}(\form)$ is contained in $h^{-1}(H)$, it follows from Proposition~\ref{prop:int} that
 $$q^{-\wt_{\form}(y)-\dim(X)}|(F_y^o\setminus h^{-1}(A\cup H))(k)| \leq \int_{\cU_y(R)}|\form|.$$
Now it suffices to observe that
$$|(F_y^o\setminus h^{-1}(A\cup H))(k)|=(q-1)^{\dim(X)-\dim(E_y)}|(E_y^o\setminus (A\cup H))(k)|.$$

  The proof of~\eqref{it:mainwild} is almost identical: by base change to a finite separable extension of $K$, we can reduce to the case where $X$ has a log smooth proper $R$-model $\cX$ such that $\cX_k$ is reduced and all its strata are geometrically connected, see Corollary~\ref{coro:reduced}. Then we can use exactly the same argument as before.
\end{proof}

\section{Convergence of Shilov measures}\label{sec:convshilov}
We now prove a convergence theorem that is valid also for non-local fields.
Let $K$ be a complete discretely valued field with perfect residue field $k$.
Let $X$ be a smooth and proper $K$-scheme of pure dimension, and let $\form$ be an $m$-canonical form on $X$, for some $m>0$, such that $\form$ is not identically zero on any connected component of $X$.

\subsection{Shilov measures}\label{ss:shimeas}
 For every $x$ in $\Sh(X,\form)$, we denote by $\tdeg_K(x)$ the tame degree of $X^{\an}$ at $x$.
  We have already explained after the statement of Theorem \ref{thm:shi} that $\tdeg(x)$ is equal to the degree of the algebraic closure of $K$ in $\mathscr{H}(x)$; moreover,
if $\cU$ is a weak N\'eron model of $X$ and  $C$ is the connected component of $\cU_k$ corresponding to $x$, then $\tdeg_K(x)$ is equal to the number of geometric connected components of $C$. In particular, it is finite.

 The \emph{Shilov measure} of $(X,\theta)$ is now defined as
 \begin{equation*}
   \mu_{X,\form}=\sum_{x\in \Sh(X,\form)}\deg_K(x) \delta_x,
 \end{equation*}
 in the same way as in Theorem \ref{thm:shi}.
 It only depends on the pair $(X,\form)$.
 Note that $\mu_{X,\form}=0$ when $\Sh(X,\theta)=\emptyset$.

\subsection{Base change of ramification index one}
Let $R'$ be a complete discretely valued extension of $R$, with quotient field $K'$ and residue field $k'$. Assume that $K'$ has ramification index one over $K$ and that $k'$ is separable over $k$.  Then the property of being a weak N\'eron model is preserved under base change from $R$ to $R'$, by~\cite[3.6/7]{BLR}.
  It follows immediately that $\Sh(X\otimes_K K',\form\otimes_K K')$ is the inverse image of $\Sh(X,K)$ under the projection map
  $(X\otimes_K K')^{\an}\to X^{\an}$, and that $\mu_{X,\form}$ is the pushforward of the measure $\mu_{X\otimes_K K',\form\otimes_K K'}$.
  The behavior under ramified extensions is more subtle, and will be discussed next.

\subsection{Convergence theorem}
Now assume that $X$ has an $snc$-model $\cX$ over $R$.
 By Proposition~\ref{prop:sklogmodel}, the Kontsevich--Soibelman skeleton $\Sk(X,\form)$ is a   union of faces of $\Sk(\cX)$, so that its dimension is well-defined.
 Let $K'$ be a finite extension of $K$. Set $X'=X\otimes_K K'$ and denote by $\form'$ the pullback of $\form$ to $X'$. We write $\mathrm{pr}\colon (X')^{\an}\to X^{\an}$ for the natural projection morphism. We define a measure $\lambda_{\form,K'}$ on $X^{\an}$ by
 $$\lambda_{\form,K'}=\frac{\mathrm{pr}_\ast\mu_{X',\form'}}{e(K'/K)^{\dim(\Sk(X,\form))}}$$
 where $\mu_{X',\form'}$ is the Shilov measure of $(X',\theta')$.
 If the temperate part $\Sk^t(X,\form)$ of $\Sk(X,\form)$ is non-empty (see Definition~\ref{defi:skmeasures}), we also define a measure $\lambda^t_{\form,K'}$ on $X^{\an}$ by
 $$\lambda^t_{\form,K'}=\frac{\mathrm{pr}_\ast\mu_{X',\form'}}{e(K'/K)^{\dim(\Sk^t(X,\form))}}.$$


 Let  $\mathcal{E}^a_K$ be the set of finite extensions $K'$ of $K$ in the algebraic closure $K^a$, ordered by inclusion. We denote by
   $\mathcal{E}^t_K$ the subset of $\mathcal{E}^a_K$ consisting of tamely ramified extensions.

\begin{theorem}\label{thm:main}
 Assume that $k$ is perfect, and that $X$ has an $snc$-model over $R$.
 \begin{enumerate}
 \item \label{it:maina} Let $K'$ be a  finite tame extension of $K$ of ramification index  $e=e(K'/K)$.
   Assume that there exists a $(1/e)\Z$-integral point on $\Sk(X,\form)$.
   Then the Shilov boundary $\Sh(X',\form')$ is the preimage in $(X')^{\an}$ of the set of $(1/e)\Z$-integral points in $\Sk(X,\form)$.

\item \label{it:tame}
 Assume that the temperate part $\Sk^t(X,\form)$ of the Kontsevich--Soibelman skeleton of the pair $(X,\form)$ is non-empty.
 As $K'$ ranges through $\mathcal{E}^t_K$, the measure $\lambda^t_{\form,K'}$ converges to the stable Lebesgue measure $\Leb^{\st}_{\Sk^t(X,\form)}$
 supported on $\Sk^t(X,\form)$.

 \item \label{it:wild} Assume that $X$ has a log smooth proper $R$-model.
  As $K'$ ranges through $\mathcal{E}^a_K$, the measure $\lambda_{\form,K'}$ converges to the stable Lebesgue measure $\Leb^{\st}_{\Sk(X,\form)}$
 supported on the Kontsevich--Soibelman skeleton of the pair $(X,\form)$.
 \end{enumerate}
\end{theorem}
\begin{proof}
\eqref{it:maina} We denote by $R'$ the valuation ring in $K'$.
  Let $\cX'$ be the normalization of $\cX\otimes_R R'$. Then the skeleton $\Sk(\cX')$ is the inverse image of $\Sk(\cX)$ in $(X')^{\an}$, by Proposition~\ref{prop:compat2}.  Moreover, the proof of Proposition~\ref{prop:compat2} shows that the inverse image of the set of $(1/e)\Z$-integral points on $\Sk(\cX)$ is the set of $\Z$-integral points on $\Sk(\cX')$.  Let $\cY'\to \cX'$ be a toroidal resolution of singularities, induced by a regular proper refinement of the Kato fan of $\cX$~\cite[\S10]{kato}.
 Then $\Sk(\cY')=\Sk(\cX')$, and the piecewise integral structure is preserved. By the definition of the integral structure, the $\Z$-integral points
 on $\Sk(\cY')$ are precisely the vertices corresponding to components of multiplicity one of the special fiber, or, equivalently, to the connected components of
 the special fiber of the $R$-smooth locus $\mathrm{Sm}(\cY')$. Since $\cY'$ is regular and proper over $R'$, its $R'$-smooth locus is a weak N\'eron model for $X'$, by~\cite[3.1/2]{BLR}. It follows from Proposition~\ref{prop:skbasech} that $\wt_{\form'}$ reaches its minimal value at a divisorial point $x'$ in $\Sk(\cY')$ if and only if
 $\wt_{\form}$ is minimal at the point $\mathrm{pr}(x')$ in $\Sk(\cX)$. Thus the Shilov boundary $\Sh(X',\form')$ is the preimage in $(X')^{\an}$ of the set of $(1/e)\Z$-integral points in $\Sk(X,\form)$, provided that this set is non-empty.

 ~\eqref{it:tame}
By Proposition \ref{prop:tame} and our assumption that $\Sk^t(X,\form)$ is non-empty, we know that $\Sk(X,\form)$ contains a $(1/e_0)\Z$-integral point, for some positive integer $e_0$ prime to $p$.
   Let $K'$ be a tame finite extension of $K$, of ramification index $e$ divisibile by $e_0$.
%
 Using point~\eqref{it:maina} and Proposition~\ref{prop:tamedeg}\eqref{it:fib}, we can write
$$\mathrm{pr}_\ast\mu_{X',\form'}=\sum_{x\in \Sk(X,\form)((1/e)\Z)} \tdeg_K(x)\delta_x$$ where the sum is taken over the $(1/e)\Z$-integral points in $\Sk(X,\form)$.
  We set $d=\dim \Sk^t(X,\form)$.
  By Proposition~\ref{prop:tame}, all  $(1/e)\Z$-integral points in $\Sk(X,\form)$ are contained in $\Sk^t(X,\form)$.
 The proof of Proposition~\ref{prop:tamecrit} implies that,
 when $e$ is sufficiently divisible, every
  $d$-dimensional face of $\Sk^t(X,\form)$ admits an integral $(1/e)\Z$-affine isomorphism
  to a polytope in $\R^d$.
  The Lebesgue measure on $\R^d$ is the limit of the discrete measures
$$\frac{1}{e^d}\sum_{x\in (1/e)\Z^d}\delta_x,$$
 where $e$ runs through the positive integers that are prime to $p$, ordered by divisibility.
  Therefore,
 the limit of the measures $$\frac{\mathrm{pr}_\ast\mu_{X',\form'}}{e(K'/K)^{d}}$$ over all $K'$ in $\mathcal{E}^t_K$ is precisely
$\Leb^{\st}_{\Sk^t(X,\form)}$.

\eqref{it:wild} The proof is entirely similar to that of~\eqref{it:tame}.
\end{proof}

\section{Further questions}
 Theorems~\ref{thm:padic} and~\ref{thm:main} leave some interesting open questions in the cases where log smooth proper $R$-models do not exist. Let us look at an explicit example.
 Assume that $k$ is a finite field of characteristic $3$, and let $E$ be an elliptic curve of Kodaira-N\'eron reduction type $IV$. 
  This curve is wildly ramified, so that it does not have a log smooth proper $R$-model. Let $\cE$ be the minimal snc-model of $E$ over $R$. The skeleton $\Sk(\cE)$ is depicted in Figure~\ref{fig:typeIV}. If $\form$ is any volume form on $E$, then $\Sk(E,\form)$ consists of the unique vertex of multiplicity $3$ in $\Sk(\cE)$. This vertex is a wild divisorial point, so that $\Sk^t(E,\form)$ is empty (even though $\Sk(E,\form)\cap E^t=\Sk(E,\form)$ by Proposition \ref{prop:tamecrit}). Thus the limit of the measures $\nu^t_{\form, K'}$ in Theorem~\ref{thm:padic}\eqref{it:maintame} is zero, and Theorem \ref{thm:main}\eqref{it:tame} says nothing about this example.

 It is easy to compute the Shilov measures $\mu_{E\otimes_K K',\form\otimes_K K'}$ from Section \ref{ss:shimeas}, for any tame finite extension $K'$ of $K$. We denote by $e=e(K'/K)$ the ramification index of $K'$ over $K$. By Proposition~\ref{prop:tamedeg}\eqref{it:one}, the tame degree function $\tdeg_K$ is equal to $1$ on $\Sk(\cE)$.
 Multiplying $\form$ with a suitable element in $K^{\times}$, we may assume that it extends to a generator of the relative canonical bundle on the minimal regular model of $E$.
 Then $\wt_{\form}$ vanishes at the three vertices of multiplicity $1$ in $\Sk(\cE)$, and it takes the value $-1/3$ at the unique vertex of multiplicity $3$. Since $\wt_{\form}$ is affine along the edges of $\Sk(\cE)$, we find that the image in $E^{\an}$ of the Shilov boundary $\Sh(E\otimes_K K',\form\otimes_K K')$ consists of precisely three points, namely, the three $(1/e)\Z$-integral points on $\Sk(\cE)$ that lie closest to the unique point in $\Sk(E,\form)$ (see Figure~\ref{fig:typeIV}).
  Thus, as $K'$ ranges through $\mathcal{E}^t_K$, the pushforward of $\mu_{E\otimes_K K',\form\otimes_K K'}$ to $E^{\an}$ converges to the measure $3\Leb_{\Sk(E,\form)}$, that is, the measure of mass $3$ supported at the unique point of $\Sk(E,\form)$.
 In the setting of Theorem \ref{thm:padic}\eqref{it:maintame}, we find the same limit measure $3\Leb_{\Sk(E,\form)}$ by 
   replacing the exponent $\wt_{\min}(E,\form)=-1/3$ by 
 $$\ord_{\min}(X\otimes_K K',\form\otimes_K K')=-\frac{\lfloor e/3 \rfloor}{e}$$    
    in the definition of the measure $\nu^t_{\form,K'}$ (Definition \ref{def:ram-measure}).

\begin{figure}[ht]
  \includegraphics[width=4cm]{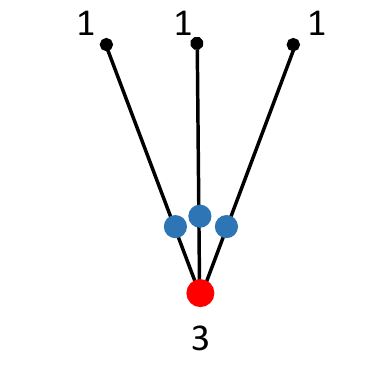}
  \caption{The skeleton of the minimal snc-model $\mathscr{E}$ of an elliptic curve $E$ of type $IV$ over $K$.
  The vertices are labelled with the multiplicities of the corresponding components in $\mathscr{E}_k$. The lattice length of each edge is equal to $1/3$.
   For any volume form $\form$ on $E$, the skeleton $\Sk(E,\form)$ consists of the unique vertex of multiplicity $3$, which is wild if $p=3$. The three marked points in the interiors of the edges form the image of the Shilov boundary $\Sh(E\otimes_K K',\form\otimes_K K')$ where $K'$ is a tame finite extension of $K$. If we set $e=e(K'/K)$, then the lattice distance from each point in $\Sh(E\otimes_K K',\form\otimes_K K')$ to the vertex of multipicity $3$ equals 
  $$\frac{1}{3}- \frac{\lfloor e/3 \rfloor}{e}.$$  }
  \label{fig:typeIV}
\end{figure}

 In view of this example, it is natural to ask if, in the settings of Theorems~\ref{thm:padic}\eqref{it:maintame} and~\ref{thm:main}\eqref{it:tame}, one can recover an interesting limit measure supported on $\Sk(X,\form)\cap X^t$, rather than the smaller set $\Sk^t(X,\form)$, by
 renormalizing the measures $\nu^t_{\form,K'}$ and by studying the asymptotic behaviour of the Shilov measures $\mu_{X\otimes_K K',\form\otimes_K K'}$ over tame finite extensions $K'$ of $K$. 
   However, the general picture seems to be more intricate than this example suggests.   It is difficult to construct examples beyond the case of elliptic curves because it is not well understood which weighted graphs may occur as skeleta of wildly ramified curves of genus $\geq 2$. Nevertheless, basic calculations on abstract graphs indicate that one can get different limit measures as $e(K'/K)$ ranges over different residue classes modulo a suitable power of  $p$. We believe that this phenomenon is related to fundamental ramification invariants introduced by Edixhoven~\cite{edix} and Chai and Yu~\cite{chai,chai-yu}  and further studied in~\cite{logjumps}. 

 Another interesting question is whether one can formulate a convergence result similar to Theorems~\ref{thm:padic} and~\ref{thm:main} without assuming the existence of an snc-model for $X$, using de Jong's results on the existence of alterations \cite{dJ}. We do not know how to use alterations to prove that the weight function $\wt_{\form}$ is bounded below on the set of divisorial points in $X^{\an}$.
 However, one can hope to retrieve the image of the skeleton of the pullback of $\form$ to a semistable alteration of $X$ by studying the asymptotic properties of Shilov measures.  All of these questions will be investigated in future work.

\end{document}